\documentclass{amsart}

\usepackage{amssymb,amscd,amsmath,hyperref,color,enumerate}
\usepackage[all,cmtip]{xy}

\def\IC{\mathbb{C}}
\def\IR{\mathbb{R}}
\def\IZ{\mathbb{Z}}
\def\IM{\mathbb{M}}

\def\IQ{\mathbb{Q}}

\def\El{\mathcal{E}\ell}

\newcommand{\Cstar}{$C^*$}

\def\CL{\mathcal{L}}
\def\CE{\mathcal{E}}
\newcommand\ICB{\mathop{\mathrm{ICB}}\nolimits}
\newcommand\IB{\mathop{\mathrm{IB}}\nolimits}
\newcommand\TM{\mathop{\mathrm{TM}}\nolimits}

\newcommand\Ext{\mathop{\mathrm{Ext}}\nolimits}

\newcommand{\SPAN}{\mathrm{span}}

\newcommand{\tr}{\mathrm{tr}}

\numberwithin{equation}{section}
\newtheorem{proposition}{Proposition}[section]
\newtheorem{lemma}[proposition]{Lemma}

\newtheorem{theorem}[proposition]{Theorem}
\newtheorem{corollary}[proposition]{Corollary}
\theoremstyle{definition}
\newtheorem{remark}[proposition]{Remark}

\newtheorem{question}[proposition]{Question}
\newtheorem{definition}[proposition]{Definition}
\newtheorem{example}[proposition]{Example}

\newtheorem{notation}[proposition]{Notation}

\begin{document}
\title[Two-sided multiplications and phantom line bundles]{The closure of two-sided
multiplications on C*-algebras and phantom line bundles}
\author[I. Gogi\'c]{Ilja Gogi\'c$^{1,2}$}
\address[1]{School of Mathematics\\Trinity College\\Dublin 2\\ Ireland}
\address[2]{Department of Mathematics\\University of Zagreb\\Bijeni\v cka
	cesta 30\\ Zagreb, 10000\\ Croatia}
\thanks{The work of the first author was supported by the Irish
Research Council under grant GOIPD/2014/7}
\author[R.~M.~Timoney]{Richard M. Timoney$^2$}
\email{ilja@maths.tcd.ie, ilja@math.hr}
\email{richardt@maths.tcd.ie}
\thanks{The work of the second author was supported
by the Science Foundation Ireland under grant 11/RFP/MTH3187.}
\date{Version \today{}}

\subjclass[2010]{Primary 46L05, 46M20; Secondary 46L07}

\keywords{$C^*$-algebra, homogeneous, two-sided multiplication, elementary operator, complex line bundle}

\begin{abstract}
For a $C^*$-algebra $A$ we consider
the problem of when the set $\TM_0(A)$ of all
two-sided multiplications $x \mapsto axb$ $(a,b \in A)$ on $A$ is norm closed,
as a subset of $\mathcal{B}(A)$. We first show that $\TM_0(A)$ is norm closed for all
prime $C^*$-algebras $A$. On the other hand, if $A\cong
\Gamma_0(\CE )$ is an $n$-homogeneous
$C^*$-algebra, where $\CE $ is the canonical $\mathbb{M}_n $-bundle over the primitive spectrum $X$ of $A$,
we show that $\TM_0(A)$ fails to be norm closed
if and only if there exists a $\sigma$-compact open subset $U$ of $X$ and a
phantom complex line subbundle
$\CL $ of $\CE $ over $U$ (i.e.
$\CL $ is not globally trivial, but is trivial on all compact subsets of $U$).
This phenomenon occurs whenever $n \geq 2$ and $X$ is a CW-complex (or a topological manifold) of dimension $3 \leq d<\infty$.
\end{abstract}

\maketitle

\section{Introduction}
\label{sec:intro}

Let $A$ be a $C^*$-algebra and let $\IB(A)$ (resp. $\ICB(A)$) denote
the set of all bounded (resp. completely bounded) maps $\phi : A \to A$ that
preserve (closed two-sided) ideals of $A$ (i.e. $\phi(I)\subseteq I$ for
all ideals $I$ of $A$). The most prominent class of maps $\phi\in \ICB(A) \subset \IB(A)$
are \textit{elementary operators}, i.e. those that can be expressed as
finite sums of \textit{two-sided multiplications} $M_{a,b} : x \mapsto
axb$, where $a$ and $b$ are elements of the multiplier algebra $M(A)$.

Elementary operators play an important role in modern quantum information and quantum computation
theory. In particular, maps $\phi : \IM_n \to \IM_n$ ($\IM_n$ are $n
\times n$ matrices over $\mathbb{C}$) of the form
$\phi=\sum_{i=1}^\ell M_{a_i^*,a_i}$ ($a_i \in \IM_n$ such that $\sum_{i=1}^\ell
a_i^*a_i=1$) represent the (trace-duals of) quantum channels, which
are mathematical models of the evolution of an `open' quantum system
(see e.g. \cite{Kribs2005}). Elementary operators also provide
ways to study the structure of \Cstar-algebras (see \cite{AraMathieuBk}).

Let $\El(A)$, $\TM(A)$ and $\TM_0(A)$ denote, respectively, the
sets of all elementary operators on $A$, two-sided multiplications
on $A$ and two-sided multiplications on $A$ with coefficients
in $A$ (i.e. $\TM_0(A) = \{ M_{a,b}: a, b \in A\}$).

The elementary operators are always dense in $\IB(A)$ in
the topology of pointwise
convergence (by \cite[Corollary 2.3]{Magajna1993PAMS}). However, more
subtle considerations enter in when one asks if $\phi \in \ICB(A)$
can be approximated pointwise by
elementary operators of cb-norm at most $\|\phi\|_{cb}$
(\cite{Magajna2009PAMS} shows that nuclearity of $A$ suffices; see
also \cite{MagajnaBelfast}).

It is an interesting problem to describe those operators $\phi \in
\IB(A)$ (resp. $\phi \in \ICB(A)$) that can be approximated
in operator norm (resp. cb-norm) by elementary operators.
Earlier works, which we cite below, revealed that this is an intricate
question in general, and can involve many and varied properties of $A$
and $\phi$.
In this paper, we show that the apparently much simpler problems of
describing the norm closures of $\TM(A)$ and $\TM_0(A)$ can have
complicated answers even for rather well-behaved \Cstar-algebras.

In some cases $\El(A) = \IB(A)$ (which implies $\El(A)
= \ICB(A)$); or $\El(A)$ is norm dense in
$\IB(A)$; or $\El(A) \subset \ICB(A)$ is dense in cb-norm.
The conditions just mentioned are in fact all equivalent for separable
\Cstar-algebras $A$.
More precisely, Magajna \cite{MagUnif2009} shows that for separable
$C^*$-algebras $A$, the property that $\El(A)$ is norm (resp.
cb-norm) dense in $\IB(A)$ (resp. $\ICB(A)$) characterizes finite
direct sums of homogeneous $C^*$-algebras with the finite type
property. Moreover, in this situation
we already have the equality $\ICB(A)=\IB(A)=\El(A)$.

It can happen that $\El(A)$ is already norm closed (or cb-norm closed).
In \cite{2011GogI,2011GogII}, the first author showed
that for a unital separable $C^*$-algebra $A$, if $\El(A)$ is
norm (or cb-norm)  closed then $A$ is necessarily subhomogeneous,
the homogeneous sub-quotients of $A$ must have the finite type
property and established further necessary conditions on $A$.
In \cite{2011GogII,2013Gog} he gave some partial converse results.

There is a considerable literature on derivations and inner derivations
of \Cstar-algebras. Inner derivations $d_a$
on a \Cstar-algebra $A$,
(i.e. those of the form $d_a(x) = ax-xa$  with $a \in M(A)$) are
important examples of elementary operators.
In \cite[Corollary 4.6]{1993Som} Somerset shows that if $A$ is unital, $\{ d_a : a
\in A\}$
is norm closed if and only if $\mathrm{Orc}(A) < \infty$, where
$\mathrm{Orc}(A)$ is a constant
defined in terms of
a certain graph structure on $\mathrm{Prim}$(A) (the primitive
spectrum of $A$).
If $\mathrm{Orc}(A) = \infty$, the structure of outer derivations
that are norm limits of inner derivations
remains undescribed. In addition, if $A$ is unital and separable, then by \cite[Theorem~5.3]{KLCR1967} and \cite[Corollary~4.6]{1993Som} $\mathrm{Orc}(A) < \infty$ if and only if the set $\{ M_{u,u^*} : u \in A, u \mbox{ unitary}\}$ of inner automorphisms is norm closed.

In \cite{2010Gog,2013Gog} the first author considered the problem of
which derivations on unital $C^*$-algebras $A$ can be cb-norm approximated by
elementary operators. By \cite[Theorem~1.5]{2013Gog} every such a
derivation is necessarily inner in a case when every Glimm ideal
of $A$ is prime. When this fails, it is possible to produce examples
which have outer derivations that are simultaneously elementary
operators (\cite[Example~6.1]{2010Gog}).

While considering derivations $d$ that are elementary operators
and/or norm limits of inner derivations, we realized that they are sometimes expressible in
the form $d = M_{a,b} - M_{b,a}$ even though they are not inner.
We have not been able to decide when all such $d$ are of this form,
but this led us to the seemingly simpler question of considering the
closures of $\TM(A)$ and $\TM_0(A)$. In this paper, we see that
nontrivial considerations enter into these questions about two-sided
multiplications. Of course the left multiplications $\{ M_{a,1} : a \in
M(A)\}$ are already norm closed, as are the right multiplications. So $\TM(A)$
is a small subclass of $\El(A)$, and seems to be the basic case to
study.
\smallskip

This paper is organized as follows.
We begin in \S \ref{sec:Preliminaries} with some generalities and an
explanation that the set of elementary operators of length at most
$\ell$ has the same completion in the operator and cb-norms (for each
$\ell \geq 1$).

In \S \ref {sec:TM-prime}, we show that for a prime \Cstar-algebra $A$, we always have
$\TM(A)$ and $\TM_0(A)$ both norm closed.

In \S \ref {sec:homog}, we recall the description of ($n$-)homogeneous
\Cstar-algebras $A$ as sections $\Gamma_0(\CE)$ of $\IM_n$-bundles $\CE $
over $X= \mathrm{Prim}(A)$ and some general results about $\IB(A)$, $\ICB(A)$
and $\El(A)$ for such $A$.

In \S \ref {sec:Fibrewise}, for
homogeneous \Cstar-algebras $A = \Gamma_0(\CE )$,
we consider subclasses $\IB_1(A)$ and $\IB_{0,1}(A)$
of $\IB(A)$ that seem (respectively) to be the most obvious
choices for the norm closure of $\TM(A)$ and of $\TM_0(A)$, extrapolating
from fibrewise restrictions on $\phi \in \TM(A)$. For each
$\phi \in \IB_1(A)$  we associate
a complex line subbundle $\CL _\phi$ of the
restriction $\CE |_U$ to an open subset $U \subseteq X =
\mathrm{Prim}(A)$,  where $U$ is determined by $\phi$ as the cozero
set of $\phi$ ($U$ identifies the fibres of $\CE $ on
which $\phi$ acts by a nonzero operator).
For separable $A$, the main result of this section
is Theorem~\ref {thm:IB1equi}, where we characterize the condition
$\phi \in \TM(A)$ in terms of triviality of the bundle
$\CL _\phi$.
We close \S \ref{sec:Fibrewise} with Remark~\ref{rem:literature}
comparing our bundle considerations
to slightly similar results in the literature
for innerness of $C(X)$-linear automorphisms when $X$ is compact, or for some
more general unital $A$.

Our final
\S \ref {sec:Closure} is the main section of this paper. For
homogeneous \Cstar-algebras $A = \Gamma_0(\CE )$, we
characterize operators $\phi$ in the norm closure of $\TM_0(A)$ as
those operators in $\IB_{0,1}(A)$ for which the associated complex
line bundle $\CL _\phi$
is trivial on each compact subset of $U$,
where $U$ is as above (Theorem~\ref{thm:CloureTM0A}).
As a consequence, we obtain that $\TM_0(A)$ fails to be norm closed
if and only if there exists a $\sigma$-compact open subset $U$ of
$X$ and a phantom complex line subbundle $\CL $
of $\CE|_U$ (i.e. $\CL $
is not globally trivial, but is trivial on each compact subset of
$U$). Using this and some algebraic topological ideas, we show that
 $\TM(A)$ and $\TM_0(A)$ both fail to be norm closed whenever $A$ is $n$-homogeneous with $n\geq 2$ and $X$ contains an open subset homeomorphic to $\IR^d$ for some $d \geq 3$ (Theorem~\ref{thm:TM-closed}).

\section{Preliminaries}
\label{sec:Preliminaries}
Throughout this paper $A$ will denote a $C^*$-algebra. By an \textit{ideal} of $A$ we always mean a closed two-sided ideal. As usual, by $Z(A)$ we denote the centre of $A$, by $M(A)$ the \textit{multiplier algebra} of $A$, and by $\mathrm{Prim}(A)$ the \textit{primitive spectrum} of $A$ (i.e. the set of kernels of all irreducible representations of $A$ equipped with the Jacobson topology).

\smallskip

Every $\phi \in \IB(A)$  is linear over $Z(M(A))$ and,
for any ideal $I$ of $A$, $\phi$ induces a map
\begin{equation}\label{eq:phiI}
\phi_I: A/I \to A/I, \quad \mbox{which sends} \quad a+I \mbox{ to } \phi(a)+I.
\end{equation}

It is easy to see that the norm (resp. cb-norm) of an operator $\phi
\in \IB(A)$ (resp. $\phi \in \ICB(A)$) can be computed via the formulae
\begin{align}\label{eq:normformula}
	\|\phi\|&=\sup\{\|\phi_P\| \ : \ P \in \mathrm{Prim}(A)\}
\quad \mbox{resp.}
\nonumber\\
\quad \|\phi\|_{cb}&=\sup\{\|\phi_P\|_{cb} \ : \ P \in \mathrm{Prim}(A)\}.
	\end{align}

\smallskip

 The \emph{length} of a non-zero elementary operator $\phi\in \El(A)$
 is the smallest positive integer $\ell=\ell(\phi)$ such that $\phi=\sum_{i=1}^\ell M_{a_i,b_i}$ for some $a_i,b_i \in M(A)$. We also define $\ell(0)=0$. We write $\El_\ell(A)$ for the
	elementary operators of length at most $\ell$. Thus $\El_1 (A)=\TM(A)$.

We will also consider the following subsets of $\TM(A)$:
\begin{eqnarray}
\TM_{cp}(A)&=&\{M_{a,a^*} : a \in M(A)\},\nonumber \\
	\mathrm{InnAut_{alg}}(A)&=&\{M_{a, a^{-1}} : a \in M(A), \ a \ \mathrm{invertible}\},  \mbox{ and}\nonumber \\
\mathrm{InnAut}(A)&=&\{M_{u,u^*} : u \in M(A), \ u \ \mathrm{unitary}\}
	\label{eqn:extraTMnotation}
\end{eqnarray}
(where cp and alg signify, respectively, completely positive and algebraic). Note that $\mathrm{InnAut}(A)=\TM_{cp}(A)\cap \mathrm{InnAut_{alg}}(A)$.

It is well known that elementary operators are completely bounded with the following estimate for their cb-norm:
\begin{equation}\label{eq:Haagest}
\left\|\sum_{i} M_{a_i,b_i}\right\|_{cb}\leq \left\|\sum_i a_i \otimes b_i\right\|_h,
\end{equation}
where $\|\cdot\|_h$ is the Haagerup tensor norm on the algebraic
tensor product $M(A) \otimes M(A)$, i.e.
\[
\|u\|_h = \inf \left\{\left\|\sum_{i} a_ia_i^*\right\|^{\frac{1}{2}}\left\|\sum_{i} b_i^*b_i\right\|^{\frac{1}{2}} \ : \ u=\sum_{i} a_i \otimes b_i\right\}.
\]

By inequality (\ref{eq:Haagest}) the mapping
\[
(M(A) \otimes M(A), \|\cdot\|_h) \to (\El(A), \|\cdot\|_{cb}) \quad \mbox{given by} \quad \sum_{i} a_i \otimes b_i \mapsto \sum_{i} M_{a_i,b_i}.
\]
defines a well-defined contraction. Its continuous extension to the completed Haagerup tensor product $M(A) \otimes_h M(A)$ is known as a \emph{canonical contraction} from $M(A) \otimes_h M(A)$ to $\mathrm{ICB}(A)$ and is denoted by $\Theta_A$.

We have the following result (see \cite[Proposition 5.4.11]{AraMathieuBk}):

\begin{theorem}[Mathieu]\label{thm:Mathieu}
$\Theta_A$ is isometric if and only if $A$ is a prime $C^*$-algebra.
\end{theorem}

The next result is a combination of
\cite[Corollary 3.8]{2003-TimoneyIJM}, (\ref{eq:normformula}) and the facts that
for $\phi = \sum_{i=1}^\ell M_{a_i, b_i}$, we have $\|\phi_\pi\|=\|\phi_{\ker \pi}\|$ and  $\|\phi_\pi\|_{cb}=\|\phi_{\ker \pi}\|_{cb}$
where for irreducible representation $\pi \colon A \to \mathcal{B}(H_\pi)$, $\phi_\pi = \sum_{i=1}^\ell M_{\pi(a_i),
\pi(b_i)} \in \El ( \mathcal{B}(H_\pi))$ (as in \cite[\S4]{2003-TimoneyIJM}).

\begin{theorem}[Timoney]\label{thm:Tim03-07} For $A$ a \Cstar-algebra and
	arbitrary $\phi \in \El(A)$ of length $\ell$ we have
$$\|\phi\|_{cb}=\|\phi^{(\ell)}\| \leq \sqrt{\ell}\|\phi\|,$$
where $\phi^{(\ell)}$ denotes the $\ell$-th amplification of $\phi$ on $M_\ell(A)$, $\phi^{(\ell)} : [x_{i,j}] \mapsto [\phi(x_{i,j})]$.

In particular, on each $\El_\ell(A)$ the metric induced by the
cb-norm is equivalent to the metric induced by the operator norm.
\end{theorem}

\section{Two-sided multiplications on prime \Cstar-algebras}
\label{sec:TM-prime}

If $A$ is a prime $C^*$-algebra, we prove here (Theorem~\ref
{thm:setMab-closed}) that $\TM(A)$ and
$\TM_0(A)$ must be closed in $\mathcal{B}(A)$.

The crucial step is the following lemma.

\begin{lemma}
	\label{lem:approxHaagerup}
Let $a,b,c$ and $d$ be norm-one elements of an operator space $V$. If
\[\|a \otimes b-c \otimes d\|_h<\varepsilon \leq 1/3,\] then there exists a complex number $\mu$ of modulus one such that
\[\max\{\|a-\mu c\|, \|b-\overline{\mu}d\|\}<6 \varepsilon.\]
\end{lemma}

\begin{proof}
	First we dispose of the simpler cases where $a$ and $c$ are
	linearly dependent or where $b$ and $d$ are linearly
	dependent.
	If $a$ and $c$ are dependent then $a = \mu c$ with $|\mu| = 1$. So
$a \otimes b - c \otimes d = c \otimes (\mu b-d)$ and $\|a \otimes b - c \otimes d\|_h = \|\mu b- d\| <
	\varepsilon < 6 \varepsilon$.
	Similarly if $b$ and $d$ are dependent, $b =
	\bar\mu d$ with $|\mu| = 1$ and $\|\bar\mu a -c\| < \varepsilon $.

	Leaving aside these cases, $a \otimes b - c \otimes d$ is a tensor of rank $2$.
	By \cite[Lemma 2.3]{1993ChatterjeeSmith}
	there is an invertible matrix $S \in \mathbb{M}_2$ such that
	\[
	\|\begin{bmatrix}a & -c \end{bmatrix}
		S\| < \varepsilon \quad
		\mbox{and} \quad
	\left\| S^{-1} \begin{bmatrix} b \\ d \end{bmatrix} \right\| <
	  \varepsilon.
	\]

Write $\alpha_{i,j}$ for the ${i,j}$ entry of $S$ and $\beta_{i,j}$ for the
${i,j}$ entry of $S^{-1}$.

Since $\alpha_{1,1} \beta_{1,1} + \alpha_{1,2} \beta_{2,1} =1$, at least one of the
absolute values $|\alpha_{1,1}|$, $| \beta_{1,1}|$, $|\alpha_{1,2}|$ or $|\beta_{2,1}|$
must be at least $1/\sqrt{2}$. We treat the four cases separately,
by very similar arguments.

\begin{description}
	\item[$|\alpha_{1,1}| \geq 1/ \sqrt{2}$]
		Now from
		\[
		\begin{bmatrix}a & -c \end{bmatrix}
			S = \begin{bmatrix}
				\alpha_{1,1} a -\alpha_{2,1} c &
				\alpha_{1,2} a -\alpha_{2,2} c
			\end{bmatrix}
		\]
		we have $\|  \alpha_{1,1} a -\alpha_{2,1} c\|  <
		\varepsilon$, so
		\[
		\left \| a - \frac{\alpha_{2,1}}{\alpha_{1,1} } c \right\|
		< \frac{\varepsilon}{|\alpha_{1,1}|}
		\leq  \sqrt{2} \varepsilon.
	\]
	Let $\lambda = \alpha_{2,1}/\alpha_{1,1}$.
	Then $$| 1 - |\lambda||=|\|a\| - |\lambda| \|c\| |\leq \|a-\lambda c\| \leq  \sqrt{2} \varepsilon,$$
	and so $ |\lambda| \in [ 1 - \sqrt{2} \varepsilon, 1+ \sqrt{2} \varepsilon]$.
	Also
	\[
		\left | \frac{\lambda}{|\lambda|} - \lambda
		\right|
		= |\lambda|\frac{|1 - |\lambda||}{|\lambda|}
		\leq (1 + \sqrt{2} \varepsilon)\frac{\sqrt{2}
		\varepsilon}{ 1 - \sqrt{2} \varepsilon} < (\sqrt{2}+2) \varepsilon
	\]
	(since $\varepsilon \leq 1/3$).
	So for $\mu =
	\frac{\lambda}{|\lambda|}$ we have $|\mu| = 1$ and
	\[
		\|a - \mu c\| \leq \|a - \lambda c\| + |\lambda - \mu|
		\|c\| < \sqrt{2} \varepsilon + (\sqrt{2}+2) \varepsilon < 5 \varepsilon.
	\]
	Then
	\begin{eqnarray*}
		a\otimes b - c \otimes d
		&=& a\otimes b - (\mu c \otimes \bar{\mu}d)
		\\
		&=& (a\otimes b) - (a \otimes \bar{\mu}d)
		+ (a\otimes \bar{\mu}d) - (\mu c \otimes \bar{\mu}d)
		\\
		&=& (a \otimes (b - \bar{\mu}d)) + ((a- \mu c)\otimes \bar{\mu} d)
	\end{eqnarray*}
	and thus
	\begin{eqnarray*}
		\| b - \bar{\mu} d \| &=& \| a \otimes (b - \bar{\mu} d) \|_h
		\\
		&\leq& \| a \otimes b - c \otimes d \|_h + \| (a- \mu c) \otimes  \bar{\mu} d\|_h
		\\
		&<& \varepsilon + 5 \varepsilon
		= 6 \varepsilon.
	\end{eqnarray*}

	\item[$|\alpha_{1,2}| \geq 1/ \sqrt{2}$]
		We start now with
		$\| \alpha_{1,2} a - \alpha_{2,2} c\| <
		\varepsilon$ and proceed in the same way (with
		$\lambda = \alpha_{2,2}/\alpha_{1,2}$).

	\item[$|\beta_{1,1}| \geq 1/ \sqrt{2}$]
		In this case we use
		\[
		S^{-1} \begin{bmatrix} b \\ d \end{bmatrix} =
			\begin{bmatrix} \beta_{1,1} b + \beta_{1,2} d\\
			\beta_{2,1} b + \beta_{2,2} d \end{bmatrix}
		\]
		and $\|\beta_{1,1} b + \beta_{1,2} d\| <  \varepsilon$,
		leading to a similar argument (with $\lambda=
		-\beta_{1,2}/\beta_{1,1}$ and $b$ taking the role of $a$).

	\item[$|\beta_{2,1}| \geq 1/ \sqrt{2}$]
		In this case use
		$\| \beta_{2,1} b + \beta_{2,2} d \| < \varepsilon$.
		\qedhere
\end{description}
\end{proof}

\begin{corollary}
	\label{cor:approxMabMcd}
	If $V$ is an operator space, the set
	$
S_1=\{a \otimes b : a,b \in V\}
$
of all elementary tensors forms a closed subset of $V \otimes_h V$.
\end{corollary}

\begin{proof}
Suppose $a_n \otimes b_n \to u \in V \otimes_h V$ (for $a_n, b_n \in V$).
If $u = 0$, certainly $u \in S_1$ and otherwise we may
assume $\|u\|_h =1$ and also that
\[
	\|a_n \otimes b_n\|_{h} = 1 =
\|a_n\| = \|b_n\| \quad (n \geq 1).
\]
Passing to a subsequence, we may suppose
\[
\|a_n \otimes b_n - a_{n+1}
\otimes b_{n+1}\|_h \leq \frac{1}{6 \cdot 2^n} \quad (n \geq 1).
\]
By Lemma \ref{lem:approxHaagerup}, we may multiply $a_n$ and $b_n$
by complex conjugate modulus one scalars chosen
inductively to get $a'_n$ and $b'_n$ such that
\[
	a_n \otimes b_n  = a'_n \otimes b'_n,
	\quad
\|a'_n - a'_{n+1}\| \leq 1/2^n \quad
\mbox{and} \quad
\|b'_n - b'_{n+1}\| \leq 1/2^n
\quad (n \geq 1).\]
In this way we find $a = \lim_{n \to
\infty} a'_n$ and $b = \lim_{n \to \infty} b'_n$ in $V$ with $u
= a \otimes b \in S_1$.
\end{proof}

\begin{question}
If $V$ is an operator space and $\ell >1$, is the set
\[
S_\ell = \left\{\sum_{i=1}^\ell a_i \otimes b_i : a_i,b_i \in V \right\}
\]
of all tensors of rank at most $\ell$ closed in $V \otimes_h V$?
In particular, can we extend Lemma \ref{lem:approxHaagerup}
as follows.

Let  $V$ be an operator space and let $\textbf{a} \odot \textbf{b}$ and $\textbf{c} \odot \textbf{d}$ be two norm-one tensors of the same
(finite) rank $\ell$ in $V \otimes_h V$, where $\textbf{a},\textbf{c}$
and $\textbf{b},\textbf{d}$ are, respectively, $1\times \ell$ and $\ell \times 1$ matrices with entries in $V$. Suppose that $\|\textbf{a} \odot
\textbf{b}-\textbf{c} \odot \textbf{d}\|_h < \varepsilon$ for some $\varepsilon >0$. Can we find absolute constants $C$ and $\delta$ (which depend only on $\ell$ and $\varepsilon$) so that $\delta \rightarrow 0$ as $\varepsilon \rightarrow 0$ with the following property:

There exists an invertible matrix $S\in \mathbb{M}_\ell$ such that
\[
\|S\|, \|S^{-1}\|\leq C, \quad  \|\textbf{a}S^{-1}-\textbf{c}\|<\delta \quad \mbox{and} \quad \|S\textbf{b}-\textbf{d}\|<\delta?
\]

\end{question}

\begin{theorem}\label{thm:setMab-closed}
	If $A$ is a prime $C^*$-algebra, then both $\TM(A)$ and $\TM_0(A)$
	are norm closed.
\end{theorem}
\begin{proof}
	By Theorem \ref{thm:Tim03-07} we may work with the cb-norm
	instead of the (operator) norm. Since $A$ is prime, by Mathieu's theorem (Theorem \ref{thm:Mathieu}) the
canonical map $\Theta : M(A) \otimes_h M(A) \to \ICB(A)$, $\Theta :a \otimes b \mapsto M_{a,b}$,
is isometric. By Corollary~\ref{cor:approxMabMcd}, the set $S_1$ of all elementary tensors in $M(A) \otimes_h M(A)$ is closed
in the Haagerup norm. Therefore, $\TM(A)=\Theta(S)$ is closed in the cb-norm.

For the case of $\TM_0(A)$, we use the same argument but work with the
restriction of $\Theta$ to $A \otimes_h A$.
\end{proof}

\begin{corollary}
	\label{cor:aut-prime}
	If $A$ is a prime $C^*$-algebra, then the sets $\TM_{cp}(A)$, $\mathrm{InnAut_{alg}}(A)$, and $\mathrm{InnAut}(A)$ (see (\ref
	{eqn:extraTMnotation}))
	are all norm closed.
\end{corollary}
\begin{proof}
Suppose that an operator $\phi$ in the norm closure of any of these sets. Then, by Theorem \ref{thm:setMab-closed} there are $b,c \in M(A)$ such that $\phi=M_{b,c}$. Let $\varepsilon >0$.

Suppose that $\phi$ is in the closure of $\TM_{cp}(A)$. By Theorem \ref{thm:Tim03-07} we may work with the cb-norm	instead of the (operator) norm. We may also assume that $\|\phi\|_{cb}=1=\|b\|=\|c\|$.
Then there is $a \in M(A)$ such that
\[
\|M_{b,c}-M_{a,a^*}\|_{cb}=\|b \otimes c-a\otimes a^*\|_h < \varepsilon
\]
(Theorem \ref{thm:Mathieu}). If $\varepsilon \leq 1/3$, by Lemma
\ref{lem:approxHaagerup} we can find a complex number $\mu$ of
modulus one such that $\|b-\mu a\|\ < 6 \varepsilon$ and
$\|c-\overline{\mu}a^*\|<6 \varepsilon$.
Then $\|b-c^*\|\leq 12 \varepsilon$. Hence $c=b^*$, so $\phi = M_{b,c} \in \TM_{cp}(A)$.

Suppose that $\phi$ is in the closure of $\mathrm{InnAut_{alg}}(A)$.
Then there is an invertible element $a \in M(A)$ such that
$\|M_{b,c}-M_{a,a^{-1}}\|< \varepsilon$. Since $A$ is an essential
ideal in $M(A)$, this implies $\|bxc-axa^{-1}\|< \varepsilon$ for
all $x \in M(A)$, $\|x\|\leq1$. Letting $x=1$ we obtain
$\|bc-1\|<\varepsilon$. Hence $c=b^{-1}$, so $\phi = M_{b,c} \in
\mathrm{InnAut_{alg}}(A)$.

$\mathrm{InnAut}(A)$ is norm closed as an intersection $\TM_{cp}(A)\cap \mathrm{InnAut_{alg}}(A)$ of two closed sets.
\end{proof}

\section{On homogeneous \Cstar-algebras}
\label{sec:homog}

We recall that a $C^*$-algebra $A$ is called \emph{$n$-homogeneous} (where $n$ is finite) if every irreducible representation of $A$ acts on an $n$-dimensional Hilbert space. We say that $A$ is \textit{homogeneous} if it is $n$-homogeneous for some $n$. We will use the following definitions and facts about homogeneous $C^*$-algebras:

\begin{remark}\label{rem:homog}
Let $A$ be an $n$-homogeneous $C^*$-algebra. By \cite[Theorem 4.2]{1951Kap} $\mathrm{Prim}(A)$ is a (locally compact) Hausdorff space. If there is no danger of confusion, we simply write $X$ for $\mathrm{Prim}(A)$.
\begin{itemize}
\item[(a)] A well-known theorem of Fell \cite[Theorem 3.2]{1961Fell},
and Tomiyama-Takesaki \cite[Theorem 5]{1961TomTak} asserts
that for any $n$-homogeneous \Cstar-algebra, $A$,
there is a locally trivial bundle
$\CE $ over $X$ with fibre $\mathbb{M}_n$ and structure group
$PU(n)=\mathrm{Aut}(\mathbb{M}_n)$ such that $A$ is isomorphic to the
$C^*$-algebra $\Gamma_0(\CE )$ of continuous sections of $\CE $ which vanish at
infinity.

Moreover, any two such algebras $A_i=\Gamma_0(\CE _i)$ with primitive
spectra $X_i$ ($i =1,2$)
are isomorphic if and only if there is a homeomorphism
$f : X_1 \to X_2$ such that $\CE _1 \cong f^*(\CE _2)$ (the pullback bundle)
as bundles over $X_1$ (see \cite[Theorem 6]{1961TomTak}). Thus, we may identify $A$ with $\Gamma_0(\CE )$.
\item[(b)] For $a \in A$ and  $t \in X$ we define $\pi_t(a) = a(t)$.
	Then, after identifying the fibre $\CE_t$ with $\IM_n$,
	$\pi_t: a \mapsto \pi_t(a)$ (for $t \in X$)
	 gives all irreducible representations of $A$ (up to the equivalence).

For a closed subset $S \subseteq X$ we define
\[
I_S =\bigcap_{t \in S} \ker
	\pi_t = \{ a \in A \ : \ a(t) =0 \ \mbox{for all} \ t \in S\}.
\]
By \cite[VII~8.7.]{FDBk} any closed two-sided ideal of $A$ is of the form $I_S$ for some closed subset $S \subset X$. Further, by the the \textit{generalized Tietze Extension Theorem} we may identify $A_S=A/I_S$ with $\Gamma_0(\CE |_S)$ (see \cite[II.~14.8. and VII~8.6.]{FDBk}). If $S=\{t\}$ we just write $A_t$.

\item[(c)] If $\phi\in \IB(A)$ and $S \subset X$ closed, we write
	$\phi_S$ for the operator $\phi_{I_S}$ on $A_S$ (see
	(\ref{eq:phiI})). If $S=\{t\}$ we just write
	$\phi_t$. If $A$ is trivial (i.e. $A=C_0(X,\IM_n )$), we will
	consider $\phi_t$ as an operator $ \colon \IM_n  \to \IM_n $
	(after identifying $A_t$ with $\IM_n$ in the obvious way).

If $U \subset X$ is open, we can regard $B= \Gamma_0(\CE|_U)$
	as the ideal $I_{ X \setminus U}$ of $A$ (by extending
	sections to be zero outside $U$) and for $\phi\in \IB(A)$, we
	then have a restriction $\phi|_U \in \IB(B)$  of $\phi$
	to this ideal (with $(\phi|_U)_t = \phi_t$ for $t \in U$).

\item[(d)] $\IB(A)=\ICB(A)$. Indeed, for $\phi \in \IB(A)$ and $t \in X$ we have $\|\phi_t\|_{cb} \leq n \|\phi_t\|$ (\cite[p.~114]{PaulsenBk}),
so by (\ref{eq:normformula}) we have $\|\phi\|_{cb} \leq n \|\phi\|$. Hence $\phi \in \ICB(A)$.

\item[(e)] Since each $a \in Z(A)$ has $a(t)$ a multiple of the identity
	in the fibre $\CE_t$ for
	each $t \in X$, we can identify $Z(A)$  with $C_0(X)$. Observe
	that $A$ is quasi-central (i.e. no primitive ideal of $A$ contains $Z(A)$).

\item[(f)] By \cite[Lemma 3.2]{MagUnif2009} we can identify $M(A)$ with $\Gamma_b(\CE )$ (the $C^*$-algebra of bounded continuous sections of $\CE $).
As usual, we will identify $Z(M(A))$ with $C_b(X)$ (using the Dauns-Hofmann theorem \cite[Theorem~A.34]{RaeWillBk}).

If $A=C_0(X, \IM_n )$, it is well known that $M(A)=C_b(X, \IM_n)=C( \beta X, \IM_n )$ \cite[Corollary 3.4]{1973APT}, where $\beta X$ denotes the Stone-\v Cech compactification.

\item[(g)]
On each fibre $\CE_t$ we can introduce an inner product $\langle \cdot,\cdot \rangle_{2}$ as follows.

Choose an open covering $\{U_\alpha\}$ of $X$ such that each $\CE|_{U_\alpha}$ is isomorphic to $U_\alpha \times \IM_n$ (as an $\IM_n$-bundle), say via isomorphism $\Phi_\alpha$.
Let
\begin{equation}
	\label{eqn:innerproductfibre}
	\langle \xi, \eta \rangle_{2} = \tr (\Phi_\alpha(\xi)
	\Phi_\alpha(\eta)^*) \qquad (\xi, \eta \in \CE_t),
\end{equation}
where $\alpha$ is chosen so that $t \in U_\alpha$ and
$\tr(\cdot)$ is the standard trace on $\IM_n$.
This is independent of the choice of $\alpha$ since all automorphisms of $\IM_n$ are inner and $\tr(\cdot)$ is invariant
under conjugation by unitaries. If $a, b \in M(A) = \Gamma_b(\CE)$ then
$t \mapsto \langle
a(t),b(t) \rangle_2$ is in $C_b(X)$.

The norm $\| \cdot \|_2$ on $\CE_t$  associated with $\langle \cdot,
\cdot \rangle_2$ satisfies
\begin{equation}\label{eqn:innerprodE}
	\|\xi\|  \leq \|\xi\|_2 \leq \sqrt{n}\|\xi\| \qquad (\xi
	\in \CE_t).
\end{equation}
In the terminology of \cite{1974Dupre}, $(\CE, \langle \cdot, \cdot
\rangle_2)$ is a (complex continuous) Hilbert bundle of rank $n^2$ with fibre
norms equivalent to the original \Cstar-norms (by (\ref{eqn:innerprodE})).

\item[(h)] $A$ is said to have the \emph{finite type property} if $\CE $ can be trivialized over some finite open cover of $X$.
By \cite[Remark~3.3]{MagUnif2009} $M(A)$ is homogeneous if and only if $A$ has the finite type property. When this fails, it is possible to have
$\mathrm{Prim}(M(A))$ non-Hausdorff \cite[Theorem~2.1]{2014ArcSomIsr}. On the other hand, $M(A)$ is always quasi-standard (see \cite[Corollary~4.10]{2011ArchSomMun}).
\end{itemize}
\end{remark}

For completeness we include a proof of the following.

\begin{proposition}
	\label{prop:ICB-is-CXbMn}
Let $X$ be a locally compact Hausdorff space and $A=C_0(X, \IM_n)$.
\begin{itemize}
\item[(a)] $\mathrm{IB}(A)$ can be identified with $C_b(X,
	\mathcal{B}(\IM_n ))$ by a mapping which sends an operator $\phi \in \IB(A)$ to the function $(t \mapsto \phi_t)$.
\item[(b)] Any $\phi \in \IB(A)$ can be written in the form
\begin{equation}\label{eq:elemrepphi}
\phi= \sum_{k, i=1}^n M_{e_{k, i},a_{k,i}},
\end{equation}
where $(e_{k,i})_{k,i=1}^n$ are standard matrix units of $\IM_n$
(considered as constant functions in $C_b(X, \IM_n )=M(A)$) and $a_{k,i}\in M(A)$ depend on
$\phi$.
\end{itemize}
Thus, we have
\[
\IB(A)=\ICB(A)=C_b(X,
	\mathcal{B}(\IM_n ))=\El(A)=\El_{n^2}(A).
\]
\end{proposition}

\begin{proof}
Let $\phi \in \IB(A)$.

(a) Suppose that the function $t \mapsto \phi_t \colon X \to
\mathcal{B}(\IM_n )$ is discontinuous at some point $t_0
	\in X$. Then there is a net $(t_\alpha)$ in $X$
	converging to $t_0$ such that $\| \phi_{t_\alpha} - \phi_{t_0}\|
	\geq \delta > 0$ for all $\alpha$. So there is $u_\alpha \in
	\IM_n $ of norm at most 1 with $\| \phi_{t_\alpha}(u_\alpha) -
	\phi_{t_0}(u_\alpha) \| \geq \delta$. Passing to a subnet we may
	suppose $u_\alpha \to u$ and then (since $\|\phi_{t_\alpha}\|
	\leq \|\phi\|$ and $\|\phi_{t_0}\| \leq \|\phi\|$) we must
	have
	\[
		\| \phi_{t_\alpha}(u) - \phi_{t_0}(u) \| > \delta/2
	\]
	for $\alpha$ large enough.

	Now choose $f \in C_0(X)$ equal to 1 on a neighbourhood of
	$t_0$ and put $a(t) = f(t) u$. We then have  $a \in A$ and
	\[
		\pi_{t_\alpha}(\phi(a)) = f(t_\alpha)
		\phi_{t_\alpha}(u) = \phi_{t_\alpha}(u)
	\]
	for large $\alpha$ and this contradicts continuity of
	$\phi(a)$ at $t_0$.

	So $t \mapsto \phi_t$ must be continuous (and also bounded by
	$\|\phi\|$).

	Conversely, assume that the function $t \mapsto \phi_t$ is continuous and
	uniformly bounded by some $M > 0$. Then for $a \in A$, $t
	\mapsto \phi_t(\pi_t(a))$ is continuous, bounded and vanishes at
	infinity, hence in $A$. So there is an associated mapping $\phi \colon
	A \to A$ which is easily seen to be bounded and linear.
	Moreover $\phi \in \IB(A)$ since all ideals of $A$ are of the
	form $I_S$ for some closed $S \subset X$.

(b) First assume that $A$ is unital, so that $X$ is compact. Then each $x \in A$ is a linear combination over $C(X)=Z(A)$ of the
	${e}_{i,j}$ and since $\phi$ is $C(X)$-linear, we have
	\[
		 x=\sum_{i,j=1}^n x_{i,j}{e}_{i,j} \ \ \Rightarrow \ \ \phi(x) = \sum_{i,j=1}^n x_{i,j} \phi({e}_{i,j}).
	\]
		We may write
	\[
		\phi({e}_{i,j}) =  \sum_{k,\ell=1}^n \phi_{i,j, k,
		\ell}
		                {e}_{k,\ell}
				=  \sum_{k=1}^n {e}_{k, i} {e}_{i, j}
				\left(\sum_{\ell=1}^n \phi_{i,j, k, \ell}
			{e}_{j, \ell} \right)
		\]
		where $\phi_{i,j, k, \ell} \in C(X)$. It follows that
	\[
		\phi(x)=  \sum_{k, i=1}^n {e}_{k, i} x
		\left(\sum_{j, \ell=1}^n \phi_{i,j, k, \ell} {e}_{j,
		\ell} \right) \ \ \Rightarrow \ \ \phi= \sum_{k, i=1}^n M_{{e}_{k, i},a_{k,i}},
	\]
where $a_{k,i}=\sum_{j, \ell=1}^n \phi_{i,j, k, \ell} {e}_{j,
		\ell} \in M(A)$.

Now suppose that $A$ is non-unital (so that $X$ is non-compact). By (a) we can identify $\phi$ with the function $t \mapsto \phi_t \colon X \to
	\mathcal{B}(\IM_n )$, which can be then uniquely extended to a continuous function $\beta X \to \mathcal{B}(\IM_n )$. This
 extension defines an operator in $\IB( C( \beta X, \IM_n )) = \IB(M(A))$, which we also denote by $\phi$. By the first part of the proof, $\phi$ can be represented as (\ref{eq:elemrepphi}).
\end{proof}

\begin{remark}
	In fact, in the case of general separable $C^*$-algebras $A$,
	Magajna \cite{MagUnif2009} establishes
	the equivalence of the following properties:
	\begin{enumerate}[(a)]
		\item $\IB(A) = \El(A)$.
		\item $\El(A)$  is norm dense in $\IB(A)$.
        \item $A$ is a finite direct sums of homogeneous
			$C^*$-algebras with the finite type property.
	\end{enumerate}
	Analyzing the arguments in \cite{MagUnif2009}, for the implication (c) $\Rightarrow$ (a) it is sufficient to assume that $X$ is paracompact.
\end{remark}

Since any $n$-homogeneous $C^*$-algebra is locally of the form $C(K,\IM_n)$ for some compact subset $K$
of $X$ with $K^\circ \neq \emptyset$, we have the following consequence of Proposition \ref{prop:ICB-is-CXbMn}:

\begin{corollary}
	\label{cor:phi-t-continuous}
If $A$ is a homogeneous $C^*$-algebra, then for any $\phi \in \IB(A)$ the function $t \mapsto \|\phi_t\|$ is continuous on $X$.
Hence the \emph{cozero set} $\mathrm{coz}(\phi) = \{ t \in X : \phi_t
\neq 0\}$ is $\sigma$-compact and open in $X$.
\end{corollary}

\section{Fibrewise length restrictions}
\label{sec:Fibrewise}

Here we consider a homogeneous \Cstar-algebra
$A = \Gamma_0(\mathcal E)$
and
operators $\phi \in \IB(A)$
such that $\phi_t$ is a two-sided multiplication on each fibre $A_t$
(with $t \in X$, and $X = \mathrm{Prim} (A)$ as usual).
We will write $\phi \in \IB_1(A)$ for this hypothesis.
For separable $A$,
the main result in this section (Theorem~\ref{thm:IB1equi})
characterizes when all such operators $\phi$ are two-sided
multiplications, in terms of triviality of
complex line subbundles of $\mathcal
E|_U$ for $U \subset X$ open.

In addition to
$\IB_1(A)$,
we introduce various subsets
$\IB_1^{\mathrm{nv}}(A)$
and $\IB_{0,1}(A)$ (Notation~\ref{not:IB-nv})
which are designed to facilitate the description of $\TM(A)$,
$\TM_0(A)$ and both of their norm closures in terms of complex line bundles.
The sufficient condition that ensures $\IB_1^{\mathrm{nv}}(A) \subset \TM(A)$
is that $X$ is paracompact with vanishing second integral \v{C}ech cohomology group $\check{H}^2(X;\IZ)$ (Corollary~\ref{cor:H2-cond}).
For $X$ compact of finite covering dimension $d$ and
$A = C(X, \IM_n ) $ we show that $\TM(A) \subsetneq \IB_1(A)$ provided $\check{H}^2(X; \IZ) \neq 0$ and $n^2 \geq (d+1)/2$
(Proposition~\ref{prop:complexH2not0}).
We get the same conclusion
$\TM(A) \subsetneq \IB_1(A)$
for $\sigma$-unital $n$-homogeneous \Cstar-algebras $A =
\Gamma_0(\CE)$ with $n \geq 2$ provided $X$ has a nonempty open subset
homeomorphic to (an open set in) $\IR^d$ with $d \geq 3$
(Corollary~\ref{cor:H2-R2sub}).

\begin{notation}
	\label{not:IB-ell}
	Let $A$ be an $n$-homogeneous $C^*$-algebra.
	For $\ell \geq 1$ we write
	\[
		\IB_\ell(A) = \{ \phi \in \IB(A) : \phi_t \in \El_\ell(A_t) \ \mbox{for all} \ t \in X \}.
	\]
\end{notation}

\begin{lemma}
	\label{lem:field-el-ops}
		Let $A=\Gamma_0(\CE )$ be a homogeneous $C^*$-algebra and  $\phi \in \IB_\ell(A)$.

If $t_0 \in X$ is such that $\phi_{t_0} \in \El_\ell(A_{t_0})
	\setminus \El_{\ell-1}(A_{t_0})$ (that is, such that $\phi_{t_0}$
	has length exactly the maximal $\ell$),
	then there are $a_1, \ldots, a_\ell, b_1, \ldots , b_\ell \in
	A$ and a compact neighbourhood $N$ of $t_0$ such that
	$\phi$ agrees with the elementary operator $\sum_{i=1}^\ell M_{a_i,b_i}$ modulo the ideal $I_N$, that is

	\[\phi(x) - \sum_{i=1}^\ell a_i x b_i \in I_N \qquad \mbox{for all} \ x \in A.\]
	Moreover, we can choose $N$ so that $\phi_t \in \El_\ell (A_{t})
	\setminus \El_{\ell-1} (A_t)$ for all $t \in N$ (that is,
	$\phi_t$ is
	of the maximal length $\ell$ for $t$ in a neighbourhood of
	$t_0$).
\end{lemma}

\begin{proof}
	Choose a compact neighbourhood $K$ of $t_0$ such that $A_K\cong
	C(K,\IM_n )$ and let $\phi_K$ be the induced operator (Remark
	\ref{rem:homog} (b), (c)).

Then, for $x \in A_K$ we have
	$\phi_K(x) = \sum_{i=1}^{n^2} c_i x d_i$ for some $c_i, d_i
	\in A_K$ (by Proposition~\ref{prop:ICB-is-CXbMn} (b)).
	Moreover we can assume that $\{ c_1(t), \ldots, c_{n^2}(t)\}$
	are linearly independent for each $t \in K$, and even
	independent of $t$.
	Since
	$(\phi_K)_{t_0} = \phi_{t_0}$ has length $\ell$, we must be
	able to
	write (in $\IM_n \otimes \IM_n$)
	\[
		\sum_{i=1}^{n^2} c_i(t_0) \otimes d_i (t_0)
		=
		\sum_{j=1}^{\ell} c'_j \otimes d'_j.
	\]
	We can choose
	 $d'_1, \ldots, d'_\ell$ to be a maximal linearly independent
	 subsequence of $d_1(t_0), \ldots, d_{n^2}(t_0)$.
	 Then, via elementary linear algebra, there is a matrix
	 $\alpha$ of size $n^2 \times \ell$ and another matrix $\beta$
	 of size $\ell \times n^2$ so that
	\[
		\begin{bmatrix} d_1(t_0) \\ \vdots \\ d_{n^2}(t_0)
		        \end{bmatrix}
		= \alpha
	\begin{bmatrix} d'_1 \\ \vdots \\ d'_\ell \end{bmatrix},
		\qquad
	\begin{bmatrix} d'_1 \\ \vdots \\ d'_\ell \end{bmatrix}
		= \beta
		\begin{bmatrix} d_1(t_0) \\ \vdots \\ d_{n^2}(t_0)
		        \end{bmatrix}
	\]
	and $\beta\alpha$ the identity. We have
	\[
	\begin{bmatrix} c'_1 & \cdots & c'_\ell \end{bmatrix}
		= \begin{bmatrix} c_1(t_0) & \cdots & c_{n^2}(t_0)
	\end{bmatrix} \alpha.
	\]

	If we define
	\[
	\begin{bmatrix} d'_1(t) \\ \vdots \\ d'_\ell(t) \end{bmatrix}
		= \beta
		\begin{bmatrix} d_1(t) \\ \vdots \\ d_{n^2}(t)
		        \end{bmatrix}
	\]
	then
	$d'_1(t), \ldots, d'_\ell(t)$ must be
	linearly independent for all $t$ in some compact neighbourhood $N$ of $t_0$.
	Thus for $t \in N$ we have (in $\IM_n \otimes \IM_n$)
	\[
		\sum_{i=1}^{n^2} c_i(t) \otimes d_i(t)
		=
		\sum_{i=1}^{n^2} c_i(t_0) \otimes d_i(t)
		=
		\sum_{j=1}^{\ell} c'_j \otimes d'_j(t).
	\]
    By Remark \ref{rem:homog} (b) we can find elements $a_j,b_j \in A$ ($1 \leq j \leq \ell$) such that $a_j(t)=c_j'$ and $b_j(t)=d_j'(t)$
    for all $t \in N$.

	Since for each $t \in N$ both of the sets $\{a_1(t), \ldots,
	a_\ell(t)\}$ and $\{b_1(t), \ldots, b_\ell(t)\}$ are linearly
	independent, we get that $\phi_t=\sum_{j=1}^\ell
	M_{a_j(t),b_j(t)}$ has length exactly $\ell$ for all $t \in
	N$ as required.
\end{proof}

\begin{corollary}
	\label{cor:field-el-ops}
	Let $A$ be a homogeneous $C^*$-algebra and $\phi \in \mathrm{IB}_1(A)$. If $t_0 \in X$ is such that $\phi_{t_0} \neq 0$ then there is a compact neighbourhood $N$ of $t_0$ and $a,b \in A$ such that $a(t) \neq 0$ and $b(t)\neq 0$ for all $t \in N$ and
	$\phi$ agrees with $M_{a,b}$ modulo the ideal $I_N$.
\end{corollary}

\begin{remark}
	\label{rem:normalise}
	Let $A=\Gamma_0(\CE)$ be a homogeneous $C^*$-algebra, $a, b \in M(A)=\Gamma_b(\CE)$ and  $\phi =M_{a,b}$.

	We may replace $a$ and $b$ by
	\[
		t \mapsto \sqrt{\frac{\|b(t)\|}{\|a(t)\|}} a(t)
		\quad \mbox{and} \quad
		t \mapsto \sqrt{\frac{\|a(t)\|}{\|b(t)\|}} b(t)
	\]
	without changing $\phi$ so as
	to ensure that $\|a(t)\| = \|b(t)\|$ for each $t \in
	X$ and that $\|\phi_t\| = \|a(t)\|^2 =  \|b(t)\|^2$ for $t
		\in X$.
\end{remark}

\begin{notation}
	\label{not:IB-nv}
	Let $A$ be a homogeneous $C^*$-algebra.
	We write
	\[
		\IB_1^{\mathrm{nv}} (A) = \{ \phi \in \IB(A) \ : \ 0 \neq \phi_t \in
		\TM(A_t) \ \mbox{for all} \ t \in X \}
	\]
	(where $\mathrm{nv}$ signifies nowhere-vanishing).

We also use
\begin{eqnarray*}
	\IB_0(A)&=&\{ \phi \in \IB(A) \ : (t \mapsto \|\phi_t\|)
	\in C_0(X)\},\\
	\IB_{0,1} (A) &=& \IB_0(A) \cap \IB_1(A) ,
	\\
	\IB^{\mathrm{nv}}_{0,1}(A) &=& \IB_1^{\mathrm{nv}} (A) \cap \IB_0(A) \mbox{, and } \\
    \TM^{\mathrm{nv}}(A)&=&\TM(A) \cap \IB_1^{\mathrm{nv}} (A).
\end{eqnarray*}

	By Remark~\ref{rem:normalise}, $\TM_0(A) = \TM(A) \cap \IB_0(A)$.
\end{notation}

\begin{proposition}
	\label{prop:bundle}
	Let $A=\Gamma_0(\CE )$ be a homogeneous $C^*$-algebra and suppose $\phi \in \IB_1^{\mathrm{nv}} (A)$. Then there is a canonically associated complex line subbundle
	$\CL _\phi$ of $\CE $
	with the property that
$$\phi \in \TM(A) \quad \Longleftrightarrow \quad \CL _\phi \ \mbox{is a
	trivial bundle}.$$
\end{proposition}

\begin{proof}
	By Corollary~\ref{cor:field-el-ops}, locally $\phi$ is a two-sided
	multiplication. That is, given $t_0 \in X$ there is a
	compact neighbourhood $N$ of $t_0$ and $a , b \in A$
	such that $\phi_t = M_{a(t), b(t)}$ for all $t \in N$.

	We define
\[ \CL _\phi \cap (\CE |_N) \quad
	\mbox{to be} \quad \{(t, \lambda a(t))  :  t \in N, \lambda \in \IC\}.
\]

	Then $\CL _\phi $ is well-defined since if $N'$ is another
	neighbourhood of a possibly different $t_0' \in X$ and $a', b'
	\in A$ have $\phi_t = M_{a'(t), b'(t)}$ for all $t \in N'$,
	then there is $\mu(t) \in \IC \setminus \{0\}$ such that $a'(t) = \mu(t) a(t)$ for $t \in N \cap N'$.

	The definition we gave of $\CL _\phi \cap (\CE |_N)$ shows
	that $\CL _\phi $ is a locally trivial complex line subbundle of $\CE $.
	The map
	
	\[
	\colon N \times \IC \to \CL _\phi \cap (\CE |_N) \quad \mbox{given by} \quad (t, \lambda) \mapsto  (t, \lambda a(t)).
	\]
	provides a local trivialization.

If 	$\phi \in \TM(A)$, then clearly $\CL _\phi$ is a trivial bundle.
Conversely, If $\CL _\phi$ is a trivial bundle, choose a continuous nowhere vanishing
section $s : X \to \CL _\phi$. Then for any neighbourhood $N$ as
	above there is a continuous map $\zeta \colon N \to \IC \setminus \{0\}$ such
	that $a(t) = \zeta(t) s(t)$. If we define $s' \colon X
	\to \CE $ by $s'(t) = (1/\zeta(t)) b(t)$ for $t \in N$,
	then we have $s, s' \in \Gamma(\CE )$
	well-defined and $\phi_t(x(t)) = s(t)x(t)s'(t)$ for all $x \in
	A$. Normalizing $s$ and $s'$ as in Remark~\ref{rem:normalise},
	we get $c,d \in \Gamma_b(\CE ) = M(A)$ (Remark \ref{rem:homog} (f)) with
	$\phi = M_{c, d}$.
\end{proof}

\begin{notation}
	If
	$A = \Gamma_0(\CE)$ is homogeneous and
	$\phi \in \IB_1(A)$,
	we consider the cozero set
	$U = \mathrm{coz}(\phi)$
	 (open by
	Corollary~\ref{cor:phi-t-continuous}) and then, for
	$B = \Gamma_0(\CE |_U)$,
	$\phi|_U \in \IB_{0,1}^{\mathrm{nv}} (B)$
	(see Remark \ref{rem:homog} (c)).
	We occasionally
	use $\CL_\phi$ for the subbundle $\CL_{\phi|_U}$ of $\CE |_U$.
\end{notation}

\begin{proposition}
	\label{prop:bundletophi}
	Let $A=\Gamma_0(\CE )$ be an $n$-homogeneous $C^*$-algebra such that $X$ is $\sigma$-compact. If $\CL$ is a complex line subbundle
	of $\CE $, then there is $\phi \in \IB_{0,1}^{\mathrm{nv}} (A)$ with
	$\CL _\phi = \CL $.
\end{proposition}

\begin{proof}
Let $\langle \cdot, \cdot \rangle_2$ be as in Remark \ref{rem:homog} (g).
With respect to this inner product we have a complementary subbundle $\CL ^\perp$ of $\CE $ such that $\CL  \oplus \CL ^\perp
	= \CE $.

Note that, by local compactness, $X$ has a base consisting of
$\sigma$-compact open sets. (If $t_0 \in U \subset X$ with $ U$ open,
choose a compact neighborhood $N$ of $t_0$ contained in $U$ and a
function $f \in C_0(X)$ supported in $N$ with $f(t_0) = 1$. Take $V =
\{ t \in X : |f(t)| > 0\}$.)

Since $X$ is $\sigma$-compact (and since every $\sigma$-compact
space is Lindel\"{o}f), we can find a countable open cover
$\{U_i\}_{i=1}^\infty$ of $X$ such that each restriction $\CE
|_{U_i}$ is trivial and each $U_i$ is $\sigma$-compact.
Then we can find $n^2$
norm-one sections
$(e_j^i)_{j=1}^{n^2}$ of $\Gamma_0(\CE |_{U_i})\cong C_0(U_i, \IM_n)$ such that
\[ \SPAN\{e^i_1(t), \cdots,  e^i_{n^2}(t)\} = \CE _t \cong \IM_n \quad
	\mbox{for all } t \in U_i.
\]
By extending outside $U_i$  with $0$ we may assume that $e_j^i$ are globally defined, so that $e^i_j \in A$.
Define $f_j^i(t)$ as the orthogonal projection of  $e_j^i$ into the
fibre $\CL_t$, so that $f_j^i \in A$.
We define
\[
\phi : A \to A \quad \mbox{by} \quad \phi =\sum_{i=1}^\infty \frac{1}{2^i} \left(\sum_{j=1}^{n^2}  M_{f_j^i, (f_j^i)^*}\right).
\]
Note that $\phi \in \IB_0(A)$ as a sum of an absolutely convergent
series of operators in $\IB_0(A)$ (and $\IB_0(A)$ is norm closed).
We claim that $\phi \in \IB_{0,1}^{\mathrm{nv}}(A)$ and $\CL _\phi
= \CL $. Indeed, for an arbitrary point $t \in X$ choose a norm-one
(in \Cstar-norm)
vector $s\in \CL _t$. Then there are scalars
$\lambda_j^i$ with $f_j^i(t) = \lambda_j^i \cdot s$ and $|\lambda_j^i|
= \|f_j^i(t)\|\leq \sqrt n\|e_j^i(t)\|=\sqrt{n}$ (by
(\ref{eqn:innerprodE})). Then
\[
\phi_t = \left(\sum_{i=1}^\infty \frac{1}{2^i} \left(\sum_{j=1}^{n^2} |\lambda_j^i|^2 \right)\right) \cdot M_{s,s^*}.
\]
This shows that $\phi \in \IB_{0,1}^{\mathrm{nv}}(A)$ and that for all $t \in X$ we have $\phi_t=M_{a(t),a^*(t)}$ for some $a(t)\in \CL_t$.
By the proof of Proposition \ref{prop:bundle} we conclude $\CL_\phi=\CL$.
\end{proof}

\begin{remark}\label{rem:subbundlematrix}
Let $\CL$  be a complex line bundle over a locally compact Hausdorff space $X$.
\begin{itemize}
\item[(a)] $\CL$ is isomorphic to a subbundle of some $\IM_2$-bundle $\CE$.
Indeed, let $\mathcal{F}=\CL \oplus (X \times \IC)$. Then $\mathcal{E}=\mathrm{Hom}(\mathcal{F},\mathcal{F})=\mathcal{F} \otimes \mathcal{F}^*$ is
an $\IM_2$-bundle with the desired property (see \cite[Example 3.5]{2007Phillips}).

Further, if $X$ is $\sigma$-compact, then $A=\Gamma_0(\CE)$ (with $\CE$ as above) is an example of a $2$-homogeneous $C^*$-algebra with $\mathrm{Prim}(A)=X$ that allows an operator $\phi \in \IB_{0,1}^{\mathrm{nv}}(A)$ such that $\CL_\phi \cong \CL$ (by Proposition \ref{prop:bundletophi}).

\item[(b)] Suppose that $\CL$ is a subbundle of a trivial bundle $X \times \IC^m$. If $p=\lceil\sqrt{m}\rceil$, then for each $n \geq p$ we can regard $\CL$ as a subbundle of a trivial matrix bundle $X \times \IM_n$, using some linear embedding $\IC^m \hookrightarrow \IM_n$.
\end{itemize}
\end{remark}

\begin{remark}\label{rem:para}
If the space $X$ is paracompact, it is well-known that locally
trivial complex line bundles over $X$ are classified by the homotopy
classes of maps from $X$ to $\IC P^\infty$ and/or by the elements
of the second integral \v Cech cohomology $\check{H}^2(X;\IZ)$ (see e.g.
\cite[Corollary 3.5.6 and Theorem 3.4.7]{HusBk}
and \cite[Proposition 4.53 and Theorem
4.42]{RaeWillBk}.) By \cite{HusBk}, we know that complex line bundles over $X$ are pullbacks of the canonical bundle over $\IC P^\infty$ (via a map from $X$ to $\IC P^\infty$).
\end{remark}

In light of Proposition \ref{prop:bundle} and Remark \ref{rem:para}, for a given a homogeneous $C^*$-algebra $A=\Gamma_0(\CE )$ we  define a map
\begin{equation}\label{eq:theta}
\theta : \IB_1^{\mathrm{nv}}(A)\to \check{H}^2(X;\IZ)
\end{equation}
which sends an operator $\phi \in \IB_1^{\mathrm{nv}}(A)$ to the corresponding class of $\CL_\phi$ in $\check{H}^2(X;\IZ)$. By Proposition \ref{prop:bundle} we have
$\theta^{-1}(0)=\TM^{\mathrm{nv}}(A)$. As a direct consequence of this observation we have:

\begin{corollary}
	\label{cor:H2-cond}
Let $A$ be a homogeneous $C^*$-algebra such that $X$ is paracompact. If $\check{H}^2(X;\IZ)=0$ then $\mathrm{IB}_1^{\mathrm{nv}}(A) = \TM^{\mathrm{nv}}(A)$.
\end{corollary}

We will now give some sufficient conditions on a trivial homogeneous $C^*$-algebra $A$ that will ensure the surjectivity of the map $\theta$. To do this, first  recall that a topological space $X$ is said to have the \emph{Lebesgue covering dimension} $d<\infty$ if $d$ is the smallest non-negative integer
with the property that each finite open cover of $X$
has a refinement in which no point of $X$
is included in more than $d+1$ elements (see e.g. \cite{EngDimThBk}). In
this case we write $d=\dim X$.

\begin{proposition}\label{prop:complexH2not0}
Let $X$ be a compact Hausdorff space with $\dim X \leq d<\infty$. For $n \geq 1$ let  $A_n = C(X,\IM_n )$. If $p= \left\lceil \sqrt{(d+1)/2} \right\rceil$ then for any $n \geq p$ the mapping  $\theta$ from (\ref{eq:theta}) is surjective. In particular, if $\check{H}^2(X;\IZ)\neq 0$, then  $\TM^{\mathrm{nv}}(A_n) \varsubsetneq \IB_1^{\mathrm{nv}}(A_n) $ for all $n \geq p$.
 \end{proposition}

To prove this will use the following fact (which may be known):
\begin{lemma}\label{lem:3compl}
Let $X$ be a CW-complex with $\dim X =d$.
Then each complex line bundle $\CL$ over $X$ is isomorphic to a
line subbundle of $X \times \IC^m$ with $m = \lceil (d+1)/2
\rceil$.
\end{lemma}

\begin{proof}
We consider $\IC P^\infty$ as a CW-complex in the usual way (see
\cite[Example 0.6]{HBk}). Let $\Psi: X \to  \IC P^\infty$ be the
classifying map of the bundle $\CL $ (Remark~\ref{rem:para}). Using
the cellular approximation theorem \cite[Theorem 4.8]{HBk} and
Remark \ref{rem:para} we may assume that the map $\Psi$ is cellular,
so that $\Psi$ takes the $k$-skeleton of $X$ to the $k$-skeleton
of $\IC P^\infty$ for all $k$.
Since $\IC P^\infty$ has one cell
in each even dimension,
$\Psi(X)$ is contained in the $d$-skeleton of $\IC P^\infty$, which
is the $(d-1)$-skeleton if $d$ is odd, and
is $\IC P^{m-1}$.

Hence $\CL $ is isomorphic to the pullback $\Psi^{*}(\gamma)$ of
the canonical line bundle
$\gamma$ on $\IC P^{m-1}$ (Remark \ref{rem:para}), a subbundle of the
trivial bundle $\IC P^{m-1} \times \IC^m$.
\end{proof}

\begin{proof}[Proof of Proposition \ref{prop:complexH2not0}]
Let $\CL$ be any complex line bundle over $X$. By the proof
of \cite[Lemma 2.3]{2007Phillips} there exists a finite complex $Y$
with $\dim Y \leq d$, a continuous function $f: X \to Y$, and a line
bundle $\mathcal{L}'$ over $Y$ such that $\CL \cong f^*(\CL')$. By Lemma
\ref{lem:3compl} we conclude that
$\CL'$ is isomorphic to a line subbundle of $Y \times \IC^m$, with $m =
\lceil (d+1)/2\rceil$. Hence, $\CL$ is isomorphic to a line subbundle
of $X \times \IC^m$. By Remark \ref{rem:subbundlematrix} (b) if $n
\geq p= \lceil\sqrt{m}\rceil=
\left\lceil \sqrt{ (d+1)/2} \right\rceil$ ($\lceil\sqrt{\lceil x\rceil}\rceil=\lceil\sqrt{x}\rceil$ for all $x \geq 0$),
we can assume that $\CL$ is already a subbundle of $X \times \IM_n$. By
the proof of Proposition \ref{prop:bundletophi} we can find an operator
$\phi \in \IB_1^{\mathrm{nv}}(A_n)$ such that $\CL_\phi=\CL$. By Remark \ref{rem:para} we conclude that the map $\theta$ is surjective. That  $\TM^{\mathrm{nv}}(A) \subsetneq \IB_1^{\mathrm{nv}}(A_n)$ ($n \geq p$) when $H^2(X;\IZ)\neq 0$ follows directly from previous observations and Proposition \ref{prop:bundle}.
\end{proof}

\begin{example}\label{ex:spheretorusklein}
Note that if
$X$ is either the $2$-sphere, the
$2$-torus or the Klein bottle, then it is well-known
that $\check{H}^2(X; \IZ) \neq 0$. In particular, if $A=C(X, \IM_n )$
($n\geq 2$) then Proposition
\ref{prop:complexH2not0} shows that
$\TM^{\mathrm{nv}}(A) \subsetneq \IB_1^{\mathrm{nv}} (A)$.
\end{example}

\begin{theorem}\label{thm:IB1equi}
Let $A=\Gamma_0(\CE )$ be a homogeneous $C^*$-algebra. Consider the following conditions:
\begin{itemize}
\item[(a)] For every open subset $U \subset X$, each complex line subbundle of $\CE|_U$ is trivial.
\item[(b)] $\mathrm{IB}_1(A)=\TM(A)$.
\item[(c)] $\mathrm{IB}_{0,1}(A)=\TM_0(A)$.
\end{itemize}
Then (a) $\Rightarrow$ (b) $\Rightarrow$ (c). If $A$ is separable, conditions (a), (b) and (c) are equivalent.
\end{theorem}

\begin{proof}
(a) $\Rightarrow$ (b):
	Assume (a) holds and $\phi \in \IB_1(A)$.
	Let $U = \mathrm{coz}(\phi)$ (open by
	Corollary~\ref{cor:phi-t-continuous}).
	By Proposition~\ref{prop:bundle} we may assume that $U\neq X$.
Let $B=\Gamma_0(\CE|_U)$ and let $\phi|_U$ be the restriction of $\phi$ to $B$.
	Then  $\phi|_U\in \IB_{0,1}^{\mathrm{nv}} (B)$. By (a), $\CL _{\phi}$ is
	trivial (on $U$) and by Proposition~\ref{prop:bundle} we have $\phi|_U \in \TM(B)$, that is $\phi|_U =
	M_{c, d}$ for some $c, d \in M(B) = \Gamma_b(\CE|_U)$. By
	Remark~\ref{rem:normalise}, we can suppose that $\|c(t)\|^2 = \|d(t)\|^2 =
		\|\phi_t\|$ for $t \in U$, so that $c,d \in B$. We can then define $a, b \in A$ by $a(t)
		= b(t) = 0$ for $t \in X \setminus U$ and, for $t \in U$,
		$a(t) = c(t)$, $b(t) =
		d(t)$. Then we have $\phi = M_{a,b} \in \TM_0(A)\subseteq \TM(A)$.

		(b) $\Rightarrow$ (c): Take intersections with $\IB_0(A)$.

\smallskip

    Now assume that $A$ is separable, so that $X$ is second-countable.

    (c) $\Rightarrow$ (b): If $\phi \in \IB_1(A)$, take a strictly positive
		function $f \in C_0(X)$
		and define $\psi \in \IB_{0,1}(A)$ by $\psi_t = f(t)^2 \phi_t$. By (c) and
		 Remark~\ref{rem:normalise} we have $\psi = M_{c,d}$ for $c, d \in A$
		 with $\|c(t)\|^2 = \|d(t)\|^2 = \|\psi_t\|$. We can define $a, b \in
		 M(A) = \Gamma_b(\CE)$ by $a(t) = c(t)/f(t)$ and
		 $b(t) = d(t)/f(t)$ to get $\phi =	 M_{a,b}\in \TM(A)$.

	(b) $\Rightarrow$ (a):
	Assume (b) holds. Let $U$ be an open subset of $X$ and $\CL $ a complex line subbundle of $\CE|_U$.
    By Proposition~\ref{prop:bundletophi} applied to $B = \Gamma_0(\CE|_U)$ ($U$ is $\sigma$-compact since $X$ is second-countable), there is
	$\psi  \in \IB_0(B)$ with
	$\CL _\psi = \CL $.
	Since $(t \mapsto \|\psi_t\|) \in C_0(U)$, we can define
	$\phi \in \IB_0(A)$ by $\phi_t = \psi_t$ for $t \in U$ and $\phi_t
	 = 0 $ for $t \in X \setminus U$.
	By (b), $\phi = M_{a,b}$ for $a, b \in M(A) = \Gamma_b(\CE)$ and then
	$a|_U$ defines a nowhere vanishing section of $\CL $.
\end{proof}

\begin{corollary}\label{cor:H2-R2sub}
	Let $A =\Gamma_0(\CE)$ be an
	$n$-homogeneous \Cstar-algebra with $n \geq 2$.
 \begin{itemize}
\item[(a)] If $X$ is second-countable with $\dim X<2$, or if $X$ is
	(homeomorphic to) a subset of a non-compact connected $2$-manifold, then
\[\mathrm{IB}_{0,1}(A)=\TM_0(A) \ \mbox{ and } \ \mathrm{IB}_{1}(A)=\TM(A).
\]
\item[(b)] If $X$ is $\sigma$-compact and contains a nonempty open subset
	homeomorphic to (an open subset of) $\IR^d$ for some $d \geq 3$, then
\[
	\mathrm{IB}_{0,1}(A)\setminus \TM_0(A) \neq \emptyset \ \mbox{ and } \ \mathrm{IB}_1(A)\setminus \TM(A)\neq \emptyset.
\]
\end{itemize}
\end{corollary}

\begin{remark}
By a $d$-manifold we always mean a second-countable topological manifold of dimension $d$.
\end{remark}

To prove this we will use the following facts (which are well-known to topologists).

\begin{remark}\label{rem:trivfindim}
 If $X$ is a metrizable space with $\dim X =d <\infty$, then any locally trivial fibre bundle over $X$ can be trivialized over some open cover of $X$ consisting of at most $d+1$ elements. This follows from Dowker's and Ostrand's theorems \cite[Theorems~3.2.1 and~3.2.4]{EngDimThBk}.
\end{remark}

\begin{lemma}\label{lem:exthomotopy}
Let $Y$ be a metrizable space with $\dim Y =d <\infty$ and let $X$ be a closed subset of $Y$. Then any map $f: X \to \IC P^\infty$ can be, up to homotopy,
continuously extended to some open neighbourhood of $X$ in $Y$.
\end{lemma}
\begin{proof}
Let $\CL$ be a complex line bundle over $X$ defined by $f$ (Remark \ref{rem:para}). By \cite[Theorem~3.1.4]{EngDimThBk},
we have $\dim X \leq \dim Y=d$. By Remark \ref{rem:trivfindim} $\CL$ can be trivialized over some open cover of $X$ consisting of (at most) $d+1$ elements. In particular, $\CL$ is determined by some map $g: X \to \IC P^d$ (see e.g. \cite[\S~3.5]{HusBk}) and by Remark \ref{rem:para} $g$ is homotopic to $f$. By \cite[Theorem V.7.1]{Hu-retracts}, finite dimensional manifolds (in particular $\IC P^d$) are ANR spaces and so by \cite[Theorem III.3.2]{Hu-retracts}, $g$ extends
(continuously) to some open neighbourhood of $X$ in $Y$.
\end{proof}

\begin{proposition}\label{prop:Bestvina}
Suppose that $X$ is a locally compact subset of a non-compact connected $2$-manifold $M$. Then $\check{H}^2(X;\IZ)=0$.
\end{proposition}
\begin{proof}

First assume that $X=M$. Then by \cite[Theorem~2.2]{NapMoh2004}, since every $2$-manifold admits a smooth structure (a classical result for which we
have failed to find a complete modern reference),
$X$ is homotopy equivalent to a CW-complex of dimension $d < 2$. Using Lemma \ref{lem:3compl} (and Remark \ref{rem:para}) we conclude that $\check{H}^2(X;\IZ)=0$.

Now let $X$ be an open subset of $M$. Since the previous
argument applies to each connected component of $X$, we again have
$\check{H}^2(X;\IZ)=0$.

If $X$ is a locally compact subset of $M$, then $X$ is open in its
closure $\overline{X}$. Let $Y=\overline{X} \setminus X$.
Then $N=M \setminus Y$ is open in $M$ and $X$ is closed in
$N$. Suppose that $\check{H}^2(X;\IZ)\neq 0$ and let $f : X \to \IC P^\infty$ be any non-null-homotopic map (Remark \ref{rem:para}). By Lemma \ref{lem:exthomotopy}
$f$ extends, up to homotopy, to a map defined on some open neighbourhood $U$ of $X$ in $M$. In particular, $\check{H}^2(U; \IZ)\neq 0$ which contradicts the second part of the proof.
\end{proof}

\begin{proof}[Proof of Corollary \ref{cor:H2-R2sub}]
For (a) it suffices to show that $\check{H}^2(U;\IZ)=0$ for all open subsets $U$
of $X$ (by Theorem \ref{thm:IB1equi} and Remark \ref{rem:para}).

By Proposition \ref{prop:Bestvina} this is true if $X$ is a subset of a non-compact connected $2$-manifold. Suppose that $X$ is second-countable with $\dim X<2$. Then for each open subset $U$ of $X$ we have $\dim U\leq \dim X$ (by the `subset theorem' \cite[Theorem~3.1.19]{EngDimThBk}), so $\check{H}^2(U;\IZ)=0$ (see e.g.  \cite[p.~94--95]{EngDimThBk}).

For (b)  we first choose an open subset $U\subset X$
for which $\CE|_U \cong U \times \IM_n$ and such that $U$
can be considered as an open set in $\IR^d$ $(d \geq 3)$. We use the simple fact that $U$
contains an open subset that has the homotopy type of the $2$-sphere $\mathbb{S}^2$.
So, replacing $U$ by such a subset, we can find a non-trivial line
subbundle $\CL$ of $U \times \IC^2$. By Remark \ref{rem:subbundlematrix} (b) we may assume that
$\CL$ is a subbundle of $U \times \IM_n \cong \CE|_U$. The
assertion now follows from the proof of Theorem \ref{thm:IB1equi}.
\end{proof}

\begin{remark}
	\label{rem:literature}
	In the literature there are somewhat similar phonomena that arise for unital $C^*$-algebras $A$
	of sections of a $C^*$-bundle over a (second-countable) compact Hausdorff space $X$.
	The question was to describe when the set $\mathrm{Aut}_{C(X)} (A)$ of all $C(X)$-linear
	automorphisms of such $A$
coincides
with the inner automorphisms of $A$ (see e.g.
\cite{Lance1969,Smith1970,PhilRae1980,PhilRaeTay1982}). For example,
if $A$ is any separable unital continuous trace $C^*$-algebra with
(primitive) spectrum $X$, there always exists an exact sequence
\[
0  \longrightarrow  \mathrm{InnAut}(A) \longrightarrow \mathrm{Aut}_{C(X)} (A) \stackrel{\eta}\longrightarrow \check{H}^2(X;\IZ)
\]
of abelian groups. In general, $\eta$ does not need to be surjective
unless $A$ is stable \cite[Theorem 2.1]{PhilRae1980}. If $A$ is
$n$-homogeneous then the image of $\eta$ is contained in the torsion
subgroup of $\check{H}^2(X;\IZ)$ \cite[2.19]{PhilRae1980}. In
particular, $A=C(\mathbb{S}^2,\IM_2 )$ shows that it can happen
that
$\TM^{\mathrm{nv}}(A)\subsetneq \IB^{\mathrm{nv}}_1(A)$  even though
$\mathrm{Aut}_{C(\mathbb{S}^2)}(A)=\mathrm{InnAut}(A)$ (since
$\check{H}^2(\mathbb{S}^2;\IZ) \cong \IZ$ is torsion free). Our
Proposition
\ref{prop:complexH2not0} shows that the map $\theta$ from
(\ref{eq:theta}) is surjective in this case.
In contrast to $\eta$, there is no obvious group structure on the domain
$\IB^{\mathrm{nv}}_1(A)$ of $\theta$.
\end{remark}

\section{Closure of $\TM(A)$ on homogeneous $C^*$-algebras}
\label{sec:Closure}

Here we continue to work with $n$-homogeneous algebras $A = \Gamma_0(\CE)$.
The class $\IB_1(A)$ considered in \S \ref{sec:Fibrewise} is rather
obviously designed to capture a restriction on the closure of $\TM(A)$
(and similarly $\IB_{0,1}(A)$ should relate to the closure of $\TM_0(A)$).
We verify right away  (Proposition~\ref{prop:IB1-normclosed}) that
$\IB_1(A)$ and $\IB_{0,1}(A)$ are indeed closed.
 However, further restrictions
on the operators $\phi$ in the closure of $\TM(A)$ arise
because triviality of the line
bundles $\CL _\psi$ associated with $\psi \in \TM(A)$ is still
present for the line bundle $\CL _\phi$
provided $U = \mathrm{coz}(\phi)$
is compact (see Corollary~\ref{cor:NormClosureTMCKMn}). If $U$ is
not compact, this triviality is evident on compact subsets of $U$ (see
Theorem~\ref{thm:CloureTM0A}, where we characterize the closure of
$\TM_0(A)$). However  $\CL _\phi$ need not be trivial globally
on $U$ (so that $\phi \notin \TM(A)$ is possible)
and this led us to
define the concept of a phantom bundle
(Definition~\ref{def:phantombundle}). The terminology
is by analogy with the well known
concept of a phantom map (see
\cite{McGibbon1995}).
Thus, in Corollary~\ref {cor:phbun}, we see that finding $\phi$ in the
norm closure of $\TM_0(A)$ with $\phi \notin \TM_0(A)$ is directly
related to finding suitable phantom complex line bundles.

For these to exist, we need $U$ to have a rather complicated
algebraic topological structure, and we find examples with $\pi_1(U)
\cong \IQ$
(Proposition~\ref{prop:new-closed}). In fact,
we can also find such
examples when $X$ contains (a copy of)
an open subset of $\IR^d$ with $d \geq 3$
and $n \geq 2$
(Theorem~\ref{thm:TM-closed}).

\begin{proposition}
	\label{prop:IB1-normclosed}
		Let $A$ be a homogeneous $C^*$-algebra.  Then $\IB_1(A)$ and $\IB_{0,1}(A)$ are norm closed subsets of $\mathcal{B}(A)$.
\end{proposition}

\begin{proof}
	If $(\phi_k)_{k=1}^\infty$ is a sequence in $\IB_1(A)$
	that converges in operator norm to $\phi \in \mathcal{B}(A)$,
	then it is clear that $\phi(I) \subset I$ for each ideal $I$ of $A$.
	Thus $\phi \in \IB(A)$.

	By (\ref{eq:normformula}) we have $\| \phi - \phi_k\| = \sup_{t \in X} \|\phi_t -
(\phi_k)_t\|$ and so $\lim_{k \to \infty} (\phi_k)_t =
	\phi_t \in \mathcal{B}(A_t)$ (for $t \in X$). Since $A_t \cong \IM_n$, invoking
	Theorem~\ref{thm:setMab-closed}, we have
	$\phi_t \in \TM(A_t)$ (for $t \in X$)
	and hence $\phi \in \IB_1(A)$.

	If $\phi_k \in \IB_{0,1}(A)$ for each $k$, then $\|(\phi_k)_t\| \to
	\|\phi_t\|$ uniformly for $t \in X$. As $(t \mapsto
	\|(\phi_k)_t\|) \in
	C_0(X)$, it follows that $(t \mapsto \|\phi_t\|) \in C_0(X)$ and so
	$\phi \in \IB_0(A)$.
\end{proof}

\begin{lemma}
	\label{lem:normalise}
	Let $A$ be a homogeneous $C^*$-algebra,
	and let $\phi \in \IB_1^{\mathrm{nv}} (A)$.
	Then there is $\psi \in \IB_1^{\mathrm{nv}} (A)$ with $\psi_t
	= \phi_t/\|\phi_t\|$ for each $t \in X$.

	Moreover $\phi \in \TM(A) \iff \psi \in \TM(A)$.
\end{lemma}

\begin{proof}
	Since $t \mapsto \| \phi_t\|$ is continuous by
	Corollary~\ref{cor:phi-t-continuous}, we can define $\psi_t
	        = \phi_t/\|\phi_t\|$ and get $\psi \in \IB(A)$ via
		local applications of
		Proposition~\ref{prop:ICB-is-CXbMn}.
		Clearly $\psi \in \IB_1^{\mathrm{nv}} (A)$.

		If $\phi = M_{a,b} \in \TM(A)$
		for $a, b \in M(A)=\Gamma_b(\CE)$, then we can normalize $a$ and $b$ as in
		Remark~\ref{rem:normalise} and then take
		$c, d \in A$ with $c(t) = a(t)/ \sqrt{ \|\phi_t\| }$,
		$d(t) = b(t)/ \sqrt{ \|\phi_t\| }$ to get
		$\psi = M_{c,d}$. So $\psi \in \TM(A)$.
		We can reverse this argument.
	\end{proof}

\begin{remark}
	\label{rem:NormClosureTM}
Let $\overline{\overline{\TM(A)}}$ denote the operator norm closure of $\TM(A)$, and
similarly for $\overline{\overline{\TM_0(A)}}$. If $A$ is homogeneous, then Proposition~\ref{prop:IB1-normclosed}
gives
$\overline{\overline{\TM(A)}} \subset \IB_1(A)$
and
$\overline{\overline{\TM_0(A)}} \subset \IB_{0,1}(A)$.
\end{remark}

\begin{proposition}
	\label{prop:NormClosureTMCKMn}
	Let $A=\Gamma_0(\CE )$ be an $n$-homogeneous $C^*$-algebra.
	Suppose that $\phi \in \overline{\overline{\TM(A)}}$ such
	that $\inf_{t \in X} \|\phi_t\| = \delta > 0$. Then $\phi \in
	\TM(A)$.
\end{proposition}

\begin{proof}
    Let $(\phi_k)_{k=1}^\infty$ be a sequence
	in $\TM(A)$ with $\lim_{k \to \infty} \phi_k =
	\phi \in \mathcal{B}(A)$.
	For $k$ large enough
	that $\|\phi_k - \phi\| < \delta/2$ we must have
	$\|(\phi_k)_t\| > \delta/2$ for each $t \in X$ (and hence
	$\phi_k \in \IB_1^{\mathrm{nv}} (A)$).
	With no loss of generality we may assume that this holds for all
	$k \geq 1$.

	Since
\[
\sup_{t \in X} |  \|(\phi_k)_t\| - \|\phi_t\| | \leq
	\sup_{t \in X}  \|(\phi_k)_t - \phi_t\| = \| \phi_k -
	\phi\|,
\]
we may use Lemma~\ref{lem:normalise} to normalise
	each $\phi_k$ and $\phi$ and assume that
\[
1 = \|\phi\| =
	\|\phi_t\| = \|(\phi_k)_t\| = \|\phi_k\|
\]
  holds for all $k \geq 1$ and $t \in X$ (and still $\lim_{k \to \infty} \phi_k =
	        \phi$).

		We now write $\phi_k = M_{a_k, b_k}$ for $a_k, b_k \in
		M(A)=\Gamma_b(\CE)$ such that $\|a_k(t)\| = \|b_k(t)\| = 1$ (for all $t
		\in X$ and all $k$). We consider the line bundle $\CL _\phi$
		associated with $\phi$ according to
		Proposition~\ref{prop:bundle} which is locally
		expressible as $\{ (t, \lambda a(t) )\}$, where
		$\phi_t = M_{a(t), b(t)}$ locally.
		We assume, as we can, that $\|a(t)\| = \|b(t)\| = 1$
		(locally).

 Let $0<\varepsilon < (18n)^{-1/2}$.

 By Remark \ref{rem:homog} (d), for $k$ suitably large (but fixed) and $t \in X$ arbitrary, we have $\|(\phi_k)_t - \phi_t\|_{cb} < \varepsilon$.
 Since, by Mathieu's theorem (Theorem \ref{thm:Mathieu}), we locally have
 \[
 \|(\phi_k)_t - \phi_t\|_{cb}=\|M_{a_k(t),b_k(t)} - M_{a(t),b(t)}\|_{cb}=\|a_k(t)\otimes b_k(t)-a(t)\otimes b(t)\|_h,
 \]
by Lemma~\ref{lem:approxHaagerup}, we can locally find a scalar $\mu_k(t)$ of modulus $1$ such that
\[ \| a_k(t) - \mu_k(t) a(t)\| < 6\varepsilon
\]
(note that $(18n)^{-1/2}<1/3$ for all $n \geq 1$).

		Consider the inner product $\langle \cdot , \cdot \rangle_2$
defined in Remark \ref{rem:homog} (g).
We claim that locally $\langle a_k(t), a(t) \rangle_2\neq 0$. Indeed, first note that (locally)
\[
		 |\langle a_k(t), a(t) \rangle_2| =
		 |\langle a_k(t), \mu_k(t) a(t) \rangle_2|
	 \]
	 and by (\ref{eqn:innerprodE})
\[
\| a_k(t)\|_2 \geq 1, \quad \| \mu_k(t) a(t)\|_2 \geq 1, \quad
	\|a_k(t) - \mu_k(t) a(t)\|_2 < 6\sqrt{n}\varepsilon.
\]
Since any two vectors $v$ and $w$ of norm at least $1$ in a
	 Hilbert space satisfy
\[
\|v - w\|_2^2 \geq \|v\|_2^2 +
	 \|w\|_2^2 - 2 |\langle v, w \rangle_2 | \geq 2(1 - |\langle v,
	 w \rangle_2 |),
\]
letting $v=a_k(t)$ and $w=\mu_k(t)a(t)$, we have (locally)
\begin{eqnarray*}
|\langle a_k(t), \mu_k(t)a(t)\rangle_2 | &\geq & 1 -
	\frac{1}{2} \| a_k(t) - \mu_k(t) a(t)\|_2^2 > 1
	- 18n\varepsilon^2 \\
&>& 0.
\end{eqnarray*}
We can therefore define $a'(t)$ locally as the normalised
	(in operator norm)
	orthogonal projection
	\[
		a'(t) = \frac{ \langle a_k(t), a(t) \rangle_2}
		{ |\langle a_k(t), a(t) \rangle_2 |} \cdot a(t).
	\]
	
	 Then $t \mapsto a'(t)$ is locally well-defined and continuous (by Remark \ref{rem:homog} (g)).
	
	As $a'(t)$ is independent of multiplying $a(t)$ by unit
	scalars, it defines a nowhere vanishing global
	section of $\CL _\phi$.
	By Proposition~\ref{prop:bundle}, we must have $\phi \in
	\TM(A)$, as required.
\end{proof}

\begin{corollary}
	\label{cor:NormClosureTMCKMn}
	Let $A$ be a unital homogeneous $C^*$-algebra.
	Then
	\[
		\overline{\overline{\TM(A)}} \cap \mathrm{IB}_1^{\mathrm{nv}}(A) \subset \TM(A).
	\]
\end{corollary}

\begin{proof}
	Let $\phi \in \overline{\overline{\TM(A)}} \cap \mathrm{IB}_1^{\mathrm{nv}}(A)$.
	Since $t \mapsto \| \phi_t\|$ is continuous
	(Corollary~\ref{cor:phi-t-continuous}) and never vanishing on
	$X$ (which is compact, as $A$ is unital), it has a minimum value $\delta > 0$.
	By Proposition~\ref{prop:NormClosureTMCKMn}, $\phi \in \TM(A)$.
\end{proof}

\begin{example}\label{ex:dim7cohnot0}
	Let $A=C(X,\IM_n )$ $(n \geq 2)$, where $X$ is any compact
	Hausdorff space with $\dim X\leq 7$ and $\check{H}^2(X;\IZ)\neq 0$.
	Then
	$\overline{\overline{\TM(A)}} \subsetneq \mathrm{IB}_1(A)$.
	Indeed, by Proposition \ref{prop:complexH2not0} there exists
	$\phi \in \IB_1^{\mathrm{nv}} (A) \setminus \TM(A)$. By
	Corollary~\ref{cor:NormClosureTMCKMn}, $\phi \not \in
	\overline{\overline{\TM(A)}}$. (Since $A$ is unital, $\TM_0(A)
	= \TM(A)$ and $\mathrm{IB}_{0,1}(A)=\mathrm{IB}_1(A)$.)
\end{example}

\begin{corollary}
	\label{cor:aut-homog}
	If $A=\Gamma_0(\CE)$ is a homogeneous $C^*$-algebra, then both
	$\mathrm{InnAut_{alg}}(A)$ and
	$\mathrm{InnAut}(A)$ (see (\ref
	{eqn:extraTMnotation})) are norm closed.
\end{corollary}

\begin{proof}
	If $M_{a, a^{-1}} \in \mathrm{InnAut_{alg}}(A)$, then for all $t \in X$ we have $\|( M_{a, a^{-1}})_t\| = \|a(t)\| \|a(t)^{-1}\| \geq 1$.
	Hence if $\phi$ is in the norm closure of $\mathrm{InnAut_{alg}}(A)$, we have $\|\phi_t\| \geq 1$ for each $t \in X$.
	By Proposition~\ref{prop:NormClosureTMCKMn}, $\phi = M_{b,c}$ for some $b,c \in M(A)$.
	Since $\phi_t(1) = 1$, $c(t)  = b(t)^{-1}$ for each $t$ and so $c = b^{-1} \in M(A) = \Gamma_b(\CE)$.

	The proof for the $\mathrm{InnAut}(A)$ is similar.
\end{proof}

\begin{remark}
The results that $\mathrm{InnAut}(A)$ is norm closed if the
\Cstar-algebra $A$ is prime or homogeneous
(in Corollaries~\ref{cor:aut-prime} and ~\ref{cor:aut-homog})
can also be deduced from \cite{1993Som,KLCR1967,2011ArchSomMun}.
To explain the deductions,
we first identify $\mathrm{InnAut}(A)$ with $\mathrm{InnAut}(M(A))$.

If $A$ is prime, then
$M(A)$ is also prime
(by \cite[Lemma~1.1.7]{AraMathieuBk}).
In particular, $\mathrm{Orc}(M(A))=1$ (in the sense
of \cite[\S2]{1993Som}), so by \cite[Corollary~4.6]{1993Som} inner
derivations of $M(A)$ are norm closed. Then \cite[Theorem~5.3]{KLCR1967}
implies that $\mathrm{InnAut}(M(A))$ is also norm closed.

If $A$ is homogeneous (or more generally quasi-central
and quasi-standard in the sense of \cite{2011ArchSomMun}), then
$M(A)$ is quasi-standard
\cite[Corollary~4.10]{2011ArchSomMun}.
Thus we have
$\mathrm{Orc}(M(A))=1$, and we may conclude as in the prime case.
\end{remark}

\begin{theorem}
	\label{thm:CloureTM0A}
Let $A=\Gamma_0(\CE )$ be a homogeneous $C^*$-algebra. For an operator $\phi \in \mathcal{B}(A)$,
the following two conditions are equivalent:
\begin{itemize}
\item[(a)] $\phi \in \overline{\overline{\TM_0(A)}}$.
\item[(b)] $\phi \in \IB_{0,1}(A)$ and
	for $U=\mathrm{coz}(\phi)$
	(open by Corollary~\ref{cor:phi-t-continuous})
	$\CL _{\phi}$ is trivial on each compact subset of $U$.
\end{itemize}
\end{theorem}

\begin{proof}
	(a) $\Rightarrow$ (b):
	Let $\phi \in \overline{\overline{\TM_0(A)}}$, so that
	$\phi  \in \IB_{0,1}(A)$ (Remark~\ref{rem:NormClosureTM}).
	For each compact subset $K \subset U$, we have $\phi_K \in
	\overline{\overline{\TM(A_K)}}$ (recall that $A_K=\Gamma(\CE|_K)$ by Remark \ref{rem:homog} (b)). By Corollary~\ref{cor:NormClosureTMCKMn}
    we have $\phi_K \in \TM(A_K)$, so that $\CL _{\phi}$ must be trivial on $K$ (by Proposition~\ref{prop:bundle}).
		
	(b) $\Rightarrow$ (a):
	Let $\phi  \in \IB_{0,1}(A)$, so that $t \mapsto \|\phi_t\|$ is in
	$C_0(X)$.

	For any
	sequence $\delta_n>0$ decreasing strictly to 0 (for instance $\delta_n = 1/n$)
	let
\[
K_n= \{ t \in X  \ : \  \|\phi_t\|\geq \delta_n  \}.
\]
Then each $K_n$ is compact,
	$K_n \subset
	K_{n+1}^\circ$ and $\bigcup_{n=1}^\infty K_n = U$.
	By Proposition~\ref{prop:bundle}, $\psi_{K_n} \in
	\TM(A_{K_n})=\TM(\Gamma(\CE|_{K_n}))$ and so there are $a_n, b_n
	\in A_{K_n}$ with $\psi_{K_n} = M_{a_n, b_n}$.
		Using Remark~\ref{rem:normalise}, we may assume
		$\|a_n(t)\| = \|b_n(t)\| = \sqrt{\|\phi_t\|}$ for $t
		\in K_n$. By Remark \ref{rem:homog} (b) we
		may extend $a_n$ to $c_n \in A$ with $c_n(t) = 0$ for
		$t \in X \setminus K_{n+1}^\circ$ and
		$\|c_n(t)\|^2 \leq
		\delta_n$ for all $t \in X \setminus K_n$. Similarly we extend $b_n$
		to $d_n \in A$ supported in $K_{n+1}^\circ$
		with $\|d_n(t)\|^2 \leq \delta_n$ for
		$t \in X \setminus K_n$. Then $(M_{c_n, d_n} - \phi)_t$ has norm at
		most $2 \delta_n$ for all $t \in X$ and hence $\lim_{n
		\to \infty} M_{c_n, d_n} = \phi$. Thus $\phi \in
		\overline{\overline{\TM_0(A)}}$.
\end{proof}

\begin{corollary}\label{cor:IB01clTM0}
For a homogeneous $C^*$-algebra $A=\Gamma_0(\CE)$ the following conditions are equivalent:
\begin{itemize}
\item[(a)] $\overline{\overline{\TM_0(A)}}=\IB_{0,1}(A)$.
\item[(b)] For each $\sigma$-compact open subset $U$ of $X$, every complex line subbundle of $\CE|_U$ is trivial on all compact subsets of $U$.
\end{itemize}
\end{corollary}
\begin{proof}
(a) $\Rightarrow$ (b): Let $U$ be a $\sigma$-compact open subset of $X$, $B=\Gamma_0(\CE|_U)$ and  $\CL$ a complex line subbundle of $\CE|_U$.
 By Proposition \ref{prop:bundletophi} we can find an operator $\phi \in \IB_{0,1}^{\mathrm{nv}}(B)$ such that $\CL_\phi=\CL$. By extending
	$\phi$ to be zero outside $U$, we may assume that $\phi \in \IB_{0,1}(A)$, so that $U=\mathrm{coz}(\phi)$.
By assumption, $\phi \in \overline{\overline{\TM_0(A)}}$, so by Theorem \ref{thm:CloureTM0A} $\CL$ is trivial on all compact subsets of $U$.

(b) $\Rightarrow$ (a): If $\phi \in \IB_{0,1}(A)$ then $U=\mathrm{coz}(\phi)$ is an open, necessarily
	$\sigma$-compact subset of $X$ (since $t \mapsto \|\phi_t\|$ is in
	$C_0(X)$). By assumption, $\CL_\phi$ is trivial on every compact subset of $U$.
Hence, $\phi \in \overline{\overline{\TM_0(A)}}$ by Theorem \ref{thm:CloureTM0A}.
\end{proof}

\begin{definition}
	\label{def:phantombundle}
A locally trivial fibre bundle $\mathcal{F}$ over a locally compact Hausdorff space $X$ is said to be a \textit{phantom bundle} if $\mathcal{F}$ is not globally trivial, but is trivial on each compact subset of $X$.
\end{definition}

\begin{corollary}\label{cor:phbun}
	Let $A=\Gamma_0(\CE )$ be a homogeneous $C^*$-algebra. Then
	$\TM_0(A)$ fails to be norm closed in $\mathcal{B}(A)$
	if and only if there exists a $\sigma$-compact
	open subset $U$ of $X$ and a phantom complex line subbundle of
	$\CE|_U$.

	If these equivalent conditions hold, then $\TM(A)$ fails to be norm closed.
\end{corollary}

\begin{proof}
	If $\TM_0(A)$ fails to be norm closed, there is $\phi \in
	\overline{\overline{\TM_0(A)}} \setminus \TM_0(A)$.
	Note that $\phi \in \IB_{0,1}(A)$ by Proposition~\ref{prop:IB1-normclosed}.
	By Theorem~\ref{thm:CloureTM0A},
	for $U= \mathrm{coz}(\phi)$ (open and $\sigma$-compact),
	$\CL _{\phi}$ is trivial on each compact subset of $U$.
	Moreover $\phi|_U \in \IB_{0,1}^{\mathrm{nv}}(B)$ for $B = \Gamma_0(\CE|_U)$.
	By Proposition~\ref{prop:bundle}, if $\CL _{\phi}$ is globally
	trivial, then $\phi|_U \in \TM(B) \cap \IB_0(B) = \TM_0(B)$. So $\phi|_U =
	M_{a,b}$ for $a, b \in B$.
	Since $B$ can be considered as an ideal of $A$ (Remark
	\ref{rem:homog} (c)), we treat $a, b \in A$.
	Hence $\phi
	= M_{a, b} \in \TM_0(A)$, a contradiction. Thus $\CL _{\phi}$ is a
	phantom bundle.

	Conversely, suppose that $U \subset X$ is open and $\sigma$-compact
	and that $\CL $ is a phantom complex line
	subbundle of $\CE|_U$. Then, taking $B =\Gamma_0(\CE|_U)$,
	Proposition~\ref{prop:bundletophi} provides $\psi \in \IB^{\mathrm{nv}}_{0,1}(B)$ with
	$\CL _\psi = \CL $. As $\CL $ is a phantom bundle, by
	Proposition~\ref{prop:bundle}, $\psi \notin \TM(B)$. We may define $\phi \in
	\IB_{0,1}(A)$ by $\phi_t = \psi_t$ for $t \in U$
	and $\phi_t = 0$ for $t \in X \setminus U$. From $\psi = \phi|_U \notin
	\TM(B)$, we have $\phi \notin \TM(A)$ but $\phi \in
	\overline{\overline{\TM_0(A)}}$ by Theorem~\ref{thm:CloureTM0A}.
\end{proof}

We now describe below a class of homogeneous $C^*$-algebras $A$ for which $\TM_0(A)$ and $\TM(A)$ both
fail to be norm closed. We first explain some preliminaries.

\begin{remark}\label{rem:EML}
Let $G$ be a group and $n$ a positive integer. Recall that a space $X$ is called an \textit{Eilenberg-MacLane} space of type
 $K(G, n)$, if it's $n$-th homotopy group $\pi_n(X)$ is isomorphic to
 $G$ and all other homotopy groups trivial. If $n > 1$ then $G$ must
 be abelian (since for all $n>1$, the homotopy groups $\pi_n(X)$  are
 abelian). We state some basic facts and examples about Eilenberg-MacLane
 spaces:
\begin{itemize}
\item[(a)] There exists a CW-complex $K(G, n)$ for any group $G$ at $n
	= 1$, and abelian group $G$ at $n > 1$. Moreover such a CW-complex is unique
up to homotopy type. Hence, by abuse of notation, it is common to denote any such space  by $K(G, n)$  \cite[p.~365-366]{HBk}.
\item[(b)] Given a CW-complex $X$, there is a bijection between its cohomology group $H^n(X; G)$ and the homotopy classes $[X, K(G, n)]$ of maps from $X$ to $K(G, n)$ \cite[Theorem 4.57]{HBk}.
\item[(c)] $K(\IZ, 2)\cong\IC P^\infty$ \cite[Example 4.50]{HBk}. In particular, by (b) and Remark \ref{rem:para}, for each CW-complex $X$ there is a bijection between $[X, K(\IZ, 2)]$ and isomorphism classes of complex line bundles over $X$.
\end{itemize}
 \end{remark}

\begin{proposition}\label{prop:K(Q,1)}
  If $X$ is a locally compact
  CW-complex of type $K(\IQ, 1)$, then every non-trivial complex line bundle over $X$ is a phantom bundle.
  Moreover there are uncountably many non-isomorphic such bundles.
\end{proposition}

\begin{proof}
 The standard model of $K(\IQ, 1)$ is the mapping telescope $\Delta$ of the sequence
\begin{equation}\label{eq:K(Q,1)}
\mathbb{S}^1 \stackrel{f_1}\longrightarrow \mathbb{S}^1 \stackrel{f_2}\longrightarrow
  \mathbb{S}^1 \stackrel{f_3}\longrightarrow \cdots,
\end{equation}
   where
   $f_n: \mathbb{S}^1 \to \mathbb{S}^1$ is given by $z \mapsto z^{n+1}$
  (see e.g. \cite[Example 1.9]{LScatBk} and \cite[Section 3.F]{HBk}).

We first consider the case when $X=\Delta$. Applying $H_1(-; \IZ)$ to the levels of the mapping telescope (\ref{eq:K(Q,1)})
gives the system
$$
\IZ \stackrel{(f_1)_*}\longrightarrow \IZ
\stackrel{(f_2)_*}\longrightarrow \IZ \stackrel{(f_3)_*}\longrightarrow
\cdots,
$$
where $(f_n)_* : \IZ \to \IZ$ is given by $k \mapsto (n+1)k$
(see \cite[Section 3.F]{HBk}). The colimit of this system is (by  \cite[Proposition
	        3.33]{HBk})
		$H_1(\Delta; \IZ)= \IQ$
		and all other integral homology groups are trivial. By the
	universal coefficient theorem for cohomology \cite[Theorem 3.2]{HBk} (see also
	\cite[\S3.F]{HBk})
each integral cohomology group of $\Delta$ is trivial, except for $\check{H}^2(\Delta;
\IZ)$ which is isomorphic to $\Ext(\IQ;\IZ)$. By \cite{Wiegold}
$\Ext(\IQ;\IZ)$ is isomorphic to the additive group of real
numbers. Hence,
by Remark \ref{rem:para},
there exists uncountably many non-isomorphic
complex line bundles over
$\Delta$. We claim that each non-trivial such bundle
$\CL $ is a phantom bundle. Indeed, for $n\geq 1$ let $\Delta_n$ denote the $n$-the level of the mapping telescope (\ref{eq:K(Q,1)}). If $K$
be an arbitrary compact subset of $\Delta$ then $K$ is contained in some $\Delta_n$. Since all $\Delta_n$'s
are homotopy equivalent to $\mathbb{S}^1$, and since $\check{H}^2(\mathbb{S}^1;\IZ)=0$, we conclude that $\CL |_{\Delta_n}$
is trivial. Then $\CL |_K$ is also trivial, since $K \subset \Delta_n$.

If $X$ is another locally compact CW-complex of
of type $K(\IQ, 1)$, then by Remark \ref{rem:EML} (a), there are maps
$f \colon \Delta \to X$ and $g \colon X \to \Delta$ such that $g \circ f$ and $f \circ g$ are
homotopic (respectively) to the identity maps (on $\Delta$ and $X$, respectively).
If $\CL$ is a non-trivial complex line bundle over $\Delta$, then $g^*(\CL)$ is non-trivial over $X$
(Remark~\ref{rem:para}).
Moreover $g^*(\CL)$ is a phantom bundle because $K \subset X$ compact implies
$g(K) \subset \Delta$ compact and $g^*(\CL)|_K$ is a restriction of
$g^*(\CL)|_{g^{-1}(g(K))} = g^*(\CL|_{g(K)})$, which is a trivial bundle. Since $g$ is a homotopy equivalence,
every non-trivial complex line bundle over $X$ must be isomorphic to $g^*(\CL)$ for some $\CL$.
\end{proof}

\begin{remark}\label{rem:CompSubsDelta}
With the same notation as in the proof of Proposition \ref{prop:K(Q,1)},
one can show that for each compact subset $K$ of $\Delta$ we have
$\check{H}^2(K;\IZ)=0$. To sketch the proof, choose an arbitrary complex line
bundle $\CL$ over $K$. Then using Lemma \ref{lem:exthomotopy} (and
Remark \ref{rem:para}) $\CL$ can be extended to an open neighbourhood
$U$ of $K$. The assertion can now be established via an argument
with triangulations of $\Delta$. There is a triangulation of $\Delta$
where $\Delta_1$ has $3$ triangles and each $\Delta_{n+1}$ has $n+3$
more triangles than $\Delta_n$. We may subdivide the triangles that touch $K$
to
get finitely many that cover $K$ and are all
contained
in $U$.  Now consider the union $T$ of the triangles that touch $K$. It is enough to show $\CL|_T$
is trivial. We can
deformation retract $T$ to a union of 1-simplices.
To do so, work on one triangle (2-cell)
at a time, starting with
any 2-cell in $\Delta_1$
with a `free' edge not in the boundary of $\Delta_1$ relative to
$\Delta_2$
(where `free' means the edge does not bound a second 2-cell).
After each step, consider the remaining 2-cells, edges and vertices.
Move on to $\Delta_2$ once all 2-cells in $\Delta_1$ are exhausted,
etc, so as to arrive at a $1$-simplex after finitely
many steps. As all complex line bundles over 1-simplices are trivial,
we have that $\CL|_T$ is trivial.
\end{remark}

In private correspondence, Mladen Bestvina informed us that we can find
phantom bundles even over some open subset of $\IR^3$, and referred
us to  \cite{1212.0128}. We outline
the construction of such a subset.

\begin{proposition}
	\label{prop:solenoid}
There exists an open subset $\Omega$ of $\IR^3$ of type $K(\IQ,1)$.
\end{proposition}

	\begin{proof}
	In \cite{1212.0128}, a construction is given
	of dense open sets $U$ in the $3$-sphere $\mathbb{S}^3$
	with fundamental groups $\pi_1(U)$
	that are large subgroups of $\IQ$. Given a sequence
	$n_i$ of natural numbers $n_i > 1$, $\pi_1(U)$ can be $\{
	p/q \in \IQ : p \in \IZ, q = \prod_{i=1}^k n_i \mbox{ for some } k\}$.
	In particular we
	will take $n_i = i+1$ and then $\pi_1(U) = \IQ$.

	The construction defines $U$ as a union of closed solid tori $U =
	\bigcup_{i=1}^\infty S_i$. For each $i$, both
	$S_i$ and the complement of its
	interior $T_i = \mathbb{S}^3 \setminus S_i^\circ$ are solid tori
	with intersection $S_i
	\cap T_i$ a ($2$-dimensional) torus. At each step,
	$T_{i+1}$ is constructed inside $T_i$ as
	an unknotted solid torus of smaller cross-sectional area
	that winds $n_i$
	times around the meridian circle of $T_i$.
	Since $T_{i+1}$ can be unfolded to a standard embedding of a
	torus via an ambient isotopy of $\mathbb{S}^3$, $S_{i+1}$ must be a
	solid torus.

Let $f: \mathbb{S}^n\to U$ be an arbitrary map. Then $f$ maps $\mathbb{S}^n$ into one of the
solid tori $S_i$ and these are homotopic to their meridian circle. In particular $\pi_n(U)=0$ for all $n>1$.
By Remark \ref{rem:EML} $U$ has the type $K(\IQ, 1)$.

Choose any point $t \in \mathbb{S}^3\setminus U$. Since $\mathbb{S}^3\setminus \{t\}$ is homeomorphic to $\IR^3$, say via the
homeomorphism $F$, then $\Omega=F(U)$ is an open subset of $\IR^3$ of
the type $K(\IQ, 1)$.
\end{proof}

\begin{proposition}
	\label{prop:new-closed}
    Let $X$ be any locally compact $\sigma$-compact CW-complex of type $K(\IQ,1)$ (e.g. $X=\Delta$). Then the $C^*$-algebra $A=C_0(X,\mathbb{M}_n)$ ($n \geq 2$) has the following property:
    \begin{quote}
     There exists an operator $\phi \in \IB_{0,1}^{\mathrm{nv}}(A) \setminus \TM(A)$ such that $\phi$ is in the norm closure of
	$\TM_{cp}(A) \cap \TM_0(A)=\{M_{a,a^*} : a \in A\}$.
\end{quote}
In particular, $\TM_0(A)$, $\TM(A)$ and $\TM_{cp}(A)$ all fail to be norm closed.

Further, if $X=\Delta$, we have $\overline{\overline{\TM_0(A)}}=\IB_{0,1}(A)$.
\end{proposition}

\begin{proof}
Choose any phantom complex line bundle $\CL$  over $X$ (Proposition \ref{prop:K(Q,1)}). Since, by Remark \ref{rem:EML} (a), $X$ has the same
homotopy type as the space $\Delta$ of Proposition \ref{prop:K(Q,1)} (which is a $2$-dimensional complex), using Remark \ref{rem:para} and Lemma \ref{lem:3compl} we may assume that $\CL $ is a subbundle of the trivial bundle $X \times \IC^2$. We also realise $\IC^2$ as a subset of $\mathbb{M}_n$  as $\{z_1 e_{1,1} + z_2e_{1,2}: z_1,z_2 \in \IC\}$, and in this way
consider $\CL $ a subbundle of $X \times \IM_n$. By the proof of Proposition \ref{prop:bundletophi} we can find two sections $a,b$ of $\CL $ vanishing at infinity (so that $a,b \in A$) such that $\SPAN\{a(t), b(t)\} = \CL _t$ for each $t \in X$.

We define a map
\[
\phi : A \to A \quad \mathrm{by}  \quad \phi=M_{a,a^*}+M_{b,b^*}.
\]
Then $\phi$ defines a completely positive elementary operator on $A$ of length at
most $2$. Clearly, $\phi_t \neq 0$ for all $t \in X$, so $\phi
\in  \IB_{0,1}^{\mathrm{nv}}(A)$.  Also, $\CL _\phi
=\CL $. Since the bundle $\CL $ is non-trivial, by Proposition
\ref{prop:bundle} we have $\phi \not \in \TM(A)$. On the other hand,
since $\CL $ is a phantom bundle, Theorem \ref{thm:CloureTM0A}
implies  $\phi \in \overline{\overline{\TM_0(A)}}$. Thus $\phi \in
\overline{\overline{\TM_0(A)}} \setminus \TM(A)$
and consequently $\phi$ has length $2$.

We have $\phi_K$ completely positive on $A_K = \Gamma(\CE|_K)$ for each compact $K \subset X$.
	Since $\CL|_K$ is a trivial bundle, $\phi_K = M_{a,b}$ for some
	$a, b \in A_K$ and we may suppose $\|a(t)\| = \|b(t)\|$ holds for all $t \in K$. It follows from positivity of $\phi_t$ that $b(t) = a(t)^*$
	(for $t \in K$). By the proof of Theorem~\ref{thm:CloureTM0A}, $\phi$ is in the norm closure of $\TM_{cp}(A) \cap \TM_0(A)$.

Now suppose that $X=\Delta$. Then by Remark \ref{rem:CompSubsDelta} $\check{H}^2(K;\IZ)=0$ for all compact subsets $K$ of $\Delta$. By Corollary \ref{cor:IB01clTM0} (and Remark \ref{rem:para}) we conclude that $\overline{\overline{\TM_0(A)}}=\IB_{0,1}(A)$.
\end{proof}

\begin{theorem}
	\label{thm:TM-closed}
	Let $A =\Gamma_0(\CE)$ be an
	$n$-homogeneous \Cstar-algebra with $n \geq 2$.
 \begin{itemize}
\item[(a)] If $X$ is second-countable with $\dim X<2$ or if $X$ is (homeomorphic to)
	a subset of a non-compact connected $2$-manifold, then both $\TM_0(A)$ and $\TM(A)$ are norm closed.
\item[(b)]
	If there is a nonempty open subset of $X$
	homeomorphic to (an open subset of) $\IR^d$ for some $d \geq 3$, then
	$\TM_0(A)$ and $\TM(A)$ both fail to be norm closed.
\end{itemize}
\begin{proof}
(a) follows from Corollary \ref{cor:H2-R2sub} (a) and Proposition \ref{prop:IB1-normclosed}.

For (b)  we first choose an open subset $U\subset X$
for which $\CE|_U$ is trivial
and
such that $U$
can be considered as an open set in $\IR^d$ $(d \geq 3)$.
Choose any open subset $V$ of $U$ that has the homotopy type of the set $\Omega$ of Proposition~\ref{prop:solenoid}.
In particular, $V$ is of type $K(\IQ,1)$, so it allows a phantom complex line bundle (Proposition \ref{prop:K(Q,1)}).
Now apply Corollary \ref{cor:phbun}.
\end{proof}
\end{theorem}

\begin{remark}
Suppose that $A=\Gamma_0(\CE)$ is a separable $n$-homogeneous $C^*$-algebra with $n \geq 2$ such that $\dim X =d<\infty$.
By Remark \ref{rem:trivfindim} (applied to an $\IM_n$-bundle $\CE$) $A$ has the finite type property. Hence, by \cite[Theorem~1.1]{MagUnif2009}, we have $\IB(A)=\El(A)$. If $X$ is either a CW-complex or a subset of a $d$-manifold, the following relations between $\TM_0(A)$, $\overline{\overline{\TM_0(A)}}$ and
$\IB_{0,1}(A)$ occur:
\begin{itemize}
\item[(a)] If $d<2$ we always have $\TM_0(A)=\overline{\overline{\TM_0(A)}}=\IB_{0,1}(A)$ (Corollary~\ref{cor:H2-R2sub} (a)).
\item[(b)] If $d=2$ we have four possibilities:
\begin{itemize}
\item[(i)] $\TM_0(A)=\overline{\overline{\TM_0(A)}}=\IB_{0,1}(A)$. This happens e.g. whenever $X$ is a subset of a non-compact connected $2$-manifold (Corollary~\ref{cor:H2-R2sub} (a)).
\item[(ii)] $\TM_0(A)=\overline{\overline{\TM_0(A)}}\subsetneq \IB_{0,1}(A)$. This happens e.g. for $A=C(X, \IM_n )$, where $X=\mathbb{S}^2$ (by Example \ref{ex:spheretorusklein} and since any proper open subset $U$ of $\mathbb{S}^2$ is homeomorphic to an open subset of $\IR^2$, so $\check{H}^2(U;\IZ)=0$ by Proposition~\ref{prop:Bestvina}).
\item[(iii)] $\TM_0(A)\subsetneq \overline{\overline{\TM_0(A)}}=\IB_{0,1}(A)$. This happens e.g. for $A=C_0(X, \IM_n )$, where $X=\Delta$ is the standard model of $K(\IQ,1)$ (Proposition \ref{prop:new-closed}).
\item[(iv)] $\TM_0(A)\subsetneq \overline{\overline{\TM_0(A)}}\subsetneq \IB_{0,1}(A)$. This happens e.g. for $A=C_0(X, \IM_n )$, where $X$ is the topological disjoint union $\mathbb{S}^2 \sqcup \Delta$ (by Proposition \ref{prop:new-closed}, Corollary \ref{cor:IB01clTM0} and Example \ref{ex:spheretorusklein}).
\end{itemize}
\item[(c)] If $d>2$ we always have $\TM_0(A)\subsetneq
	\overline{\overline{\TM_0(A)}}\subsetneq \IB_{0,1}(A)$ (by
	Theorem~\ref{thm:TM-closed} (b) and the fact that $X$ must
	contain an open subset homeomorphic to $\IR^d$ --- if $X$ is
	a subset of a $d$-manifold, this follows from \cite[Theorems
	1.7.7, 1.8.9 and 4.1.9]{EngDimThBk}).
\end{itemize}
Similar relations occur between
$\TM(A)$, $\overline{\overline{\TM(A)}}$ and  $\IB_{1}(A)$ in parts (a) and (c) of the above cases.
\end{remark}

\noindent
\textbf{Acknowledgement:}
	We are very grateful to Mladen Bestvina for his extensive and generous help with many of the topological
aspects of the paper.

\end{document}